\documentclass[preprint]{elsarticle}
\usepackage{url} %
\usepackage{amsmath,amsfonts,mathrsfs,amssymb,mathtools}
\usepackage{bbm}
\usepackage{tikz,pgfplots}
\usetikzlibrary{pgfplots.groupplots}
\usepackage{subfig}
\usepackage{graphbox}
\usepackage{float}
\usepackage{multirow}
\usepackage{booktabs}
\usepackage{amsthm}
\usepackage{algorithm}          %
\usepackage{float}
\usepackage{cleveref}
\newfloat{algorithm}{t}{lop}
\usepackage{geometry}
\usepackage{comment}
\usepackage{algpseudocode}      %
\newcommand{\bs}[1]{{\boldsymbol #1}}

\usepackage{siunitx}

\def \data{data}
\def \plots{plots}
\def \fig{fig}
\def \mitgcm{\textsf{JFNK (MITgcm)}}
\def \gascoigne{\textsf{modN.\ (Gascoigne)}}
\def \svnewton{\textsf{svN.\ (Firedrake)}}
\def \stdnewton{\textsf{stdN.\ (Firedrake)}}

\newcommand{\set}[1]{\ensuremath{\mathbb{#1}}}
\newcommand{\N}{\set{N}}

\newcommand{\boldvar}[1]{\ensuremath{\boldsymbol{#1}}}
\newcommand{\boldvec}[1]{\ensuremath{\mathbf{#1}}}

\ifdefined\vec
  \renewcommand{\vec}[1]{\boldvec{#1}}
\else
  \newcommand{\vec}[1]{\boldvec{#1}}
\fi

\newcommand{\vel}{\boldvar{v}}

\newcommand{\strainratetensor}{\dot{\boldvar{\varepsilon}}}

\newcommand{\auxtau}{\bs\tau}

\definecolor{col-co}{RGB}{255,69,0} %
\definecolor{col-id}{RGB}{30,144,255} %
\biboptions{sort&compress}

\usepackage{amsthm}

\theoremstyle{remark}

\newtheorem{theorem}{Theorem}

\begin{document}

\begin{frontmatter}
  \title{
  Robust and efficient primal-dual Newton-Krylov solvers for
  viscous-plastic sea-ice models}  
\author[courant]{Yu-hsuan Shih}
\author[maxplank]{Carolin Mehlmann}
\author[wegener] {Martin Losch}
\author[courant]{Georg Stadler}
\address[courant]{Courant Institute of Mathematical Sciences, New York
  University, New York, NY, USA}
\address[maxplank]{Otto-von-Guericke Universität, Magdeburg, Germany}
\address[wegener]{Alfred-Wegener-Institut, Helmholtz-Zentrum f{\"u}r Polar- und Meeresforschung, Bremerhaven, Germany}


\begin{abstract}
We present a Newton-Krylov solver for a viscous-plastic sea-ice
model. This constitutive relation is commonly used in climate
models to describe the material properties of sea ice.  Due to the
strong nonlinearity introduced by the material law in the momentum equation, the development of fast,
robust and scalable solvers is still a substantial challenge. In this
paper, we propose a novel primal-dual Newton linearization for the
implicitly-in-time discretized momentum equation.  Compared to
existing methods, it converges faster and more robustly with respect
to mesh refinement, and thus enables numerically converged sea-ice
simulations at high resolutions.  Combined with an algebraic multigrid-preconditioned
Krylov method for the linearized systems, which contain strongly
varying coefficients, the resulting solver scales well and can be used
in parallel.  We present experiments for two challenging test problems
and study solver performance for problems with up to 8.4 million
spatial unknowns.
\end{abstract}

\begin{keyword}
viscous-plastic rheology\sep primal--dual\sep stress--velocity Newton\sep
AMG-preconditioned Krylov\sep sea-ice\sep localization
\end{keyword}

\end{frontmatter}


\section{Introduction}\label{sec:intro}
Satellite images show that the sea-ice cover is characterized by local
linear formations, so-called linear kinematic features (LKFs) \cite{Kwok2008}. These
LKFs are associated to leads and pressure ridges (i.e., regions of low
or high sea-ice density due to diverging or converging ice motion)
and play a major role
on sea-ice formation and the ocean-ice-atmosphere interaction.  To
represent LKFs in sea-ice models, the sea-ice rheology
plays a crucial role. Most sea-ice models used in climate studies treat
sea-ice as a continuum with a viscous-plastic (VP) constitutive law \cite{Blockley2020},
although alternative rheologies have been suggested
\cite{Tsamados2013, Girard2011, Rampal2016}. These approaches are all
based on the continuum assumption, which implies that statistical
averages can be derived from a large number of ice floes.
The underlying assumptions of the continuum approach break down at high resolutions and 
there are recently suggested approaches that resolve discrete floes 
\cite{Wilchinsky2012, Herman2015,
  WestOConnorParnoEtAl21, Damsgraad2018, TuhkuriPolojarvi18}.  However, the
majority of sea-ice models in climate simulations are still using the VP rheology and
will do so in the foreseeable future \cite{Blockley2020}. 
Recent analyses demonstrated
that models based on the VP rheology reproduce scaling and
statistics of observed LKFs  \cite{Hutter2020}. However, resolving LKFs in the
viscous-plastic sea-ice model is computational very costly as LKFs
require high spatial resolution \cite{Lemieux2012,Koldunov19}. Furthermore, the VP sea-ice model is stiff and resolving
LKFs typically requires
that a non-linear solver performs a large number of iterations, and this number grows with
increasing mesh resolution.

\paragraph{Viscous-plastic sea-ice solvers}
Currently there are mainly two solution approaches for VP
sea-ice models. Firstly, and most commonly, explicit time stepping procedures for the momentum equation are used. Here, as proposed by
\cite{HunkeDukowicz1997}, pseudo-elastic terms are added to sea-ice
stresses, which leads to the elastic-viscous-plastic (EVP) method.
The additional pseudo-elastic term increases the order of the time
derivative of the coupled momentum-stress equations, making the system
less stiff and thus numerically easier to solve.
The resulting momentum equation can be
integrated explicitly, albeit using a large number of small time steps
for each outer time step of the sea-ice model.

With recent modifications to EVP \cite{Boullion2013, Lemieux2012,
  Kimmritz2015}, accurate VP solutions are possible in principle, but
because the convergence rate of the method is linear and slow for
high resolution simulations, practical solvers
limit the number of
iterations and hence the accuracy \cite{Koldunov19}. In spite of
this, EVP is the most commonly used method of solving the momentum
equation in sea-ice climate models.
Secondly, implicit methods for the momentum
equation are used. These range from the application of a few Picard
iterations per time step \cite{ZhangHibler1997}, to more sophisticated
Newton solvers with better convergence properties
\cite{Lemieux2010,Losch2014,MehlmannRichter2016newton}. 
Fully converged solutions with Newton solvers are still very expensive
for climate-type applications, which is why they have been mostly used  as reference solutions for the VP model to evaluate non-optimal solvers such as EVP \cite[e.g.,][]{Lemieux2012, KIMMRITZ2017}.

Typically, the simulation time per unknown increases substantially
with the problem size for both, available explicit and implicit solvers. 
For explicit solvers, this is due to the
fact that resolving the VP rheology requires an
increasing number of small time steps, and solver stabilization
depends on the spatial discretization and hence results in a slower convergence \cite{Kimmritz2015, Danilov2021}.
For implicit solvers, typically the
number of Newton or Picard iterations increases upon mesh
refinement, and the linearized systems become more
ill-conditioned and thus iterative linear solvers converge more slowly. Direct linear
solvers do not suffer from this limitation, but they cannot be used
for large systems due to memory limitations. Even when memory is
not a limiting factor, the solve time of direct linear solvers grows
substantially faster-than-linear with the number of unknowns.
Currently used solvers for the linearized system include relaxation methods such as
the line successive over-relaxation \cite{ZhangHibler1997,
  Lemieux2009, Losch2014} and geometric multigrid
\cite{MehlmannRichter2016mg}.


\paragraph{Approach and related work}
In this paper, we focus on implicit Newton-type solvers for the
momentum equation. We address the common growth of nonlinear and
linear iterations for increasing problems size through (1) a novel
Newton linearization, and (2) the use of algebraic multigrid (AMG)
preconditioning for the arising linear problems.  The novel Newton
iteration is based on a temporary introduction of an independent
stress variable during the linearization. This idea is  similar to the
introduction of a dual variable in primal-dual interior point methods
\cite{Wright97}. For nonlinear physics problems, similar approaches
have successfully been used for non-Newtonian fluids
\cite{RudiShihStadler20, DelosreyesGonzalez09} or total variation
image regularization \cite{HintermullerStadler06,
  ChanGolubMulet99}. To solve the large-scale linear matrix systems
arising upon linearization, we propose an AMG-preconditioned Krylov solver.
Since these systems involve
strongly varying and anisotropic coefficients as results of the
VP constitutive relation, we find it beneficial to
adjust default parameters in AMG, for which we use the parallel Hypre
library \cite{FalgoutYang02}.

\paragraph{Contributions and limitations}
In this paper, we make the following contributions.
(1) We propose an alternative Newton linearization for the sea-ice
momentum equation whose number of Newton steps is insensitive to the
mesh resolution and to regularization parameters.
(2) We propose an AMG-preconditioned Krylov method to
solve the Newton linearizations efficiently and in parallel and
present results for problems with up to 8.4 million spatial unknowns.
(3) We provide an implementation of our method in the open-source
Firedrake library and present performance comparisons with other
commonly used sea-ice solvers for benchmark problems.

Our approach also has some limitations.
(1) The proposed method requires the introduction of an additional
tensor variable. However, this variable is updates explicitly and the
linear systems to be solved in each Newton step are of the same size
as for other Newton or Picard methods.
(2) Despite substantial improvements, we do not achieve perfect weak
scalability, i.e., when the unknowns and the compute resources are
increased at the same rate, the solution time does not remain constant
but increases, although moderately. Besides known challenges to
achieving perfect scalability, this is also due to
the fact that upon mesh refinement, the VP constitutive
relation yields ever-smaller features, making the physics problem more
challenging.



\section{Governing Equations}\label{sec:equations}
Let $\Omega\subset \mathbb{R}^2$ be the spatial sea-ice domain, and
$I_t := [0, T]\subset \mathbb{R}$ the time interval over which we
consider the sea-ice evolution. Following \cite{Hibler1979}, we consider the following system of
equations for the sea-ice concentration $A(t,\bs x)\in [0,1]$, the mean
sea-ice thickness $H(t,\bs x) \geq 0$ and sea-ice velocity $\vel(t,\bs
x) \in \mathbb{R}^2$, for all $(t,\bs x)\in I_t\times \Omega$:
\begin{subequations}\label{eq:model}
\begin{alignat}{2}
    \rho_{\text{ice}}H\left(\partial_t\vel+f_c\bs e_r\times(\vel-\vel_o)\right) 
    &=\nabla\cdot\boldvar\sigma + \boldvar\tau_{\text{ocean}}(\vel,t) + \boldvar\tau_{\text{atm}}(t)&&\quad\text{ in }(0,T)\times\Omega,\label{eq:momen}\\
    \partial_t A + \nabla\cdot\left(\vel A\right)&=0&&\quad\text{ in }(0,T)\times\Omega, \label{eq:model:A}\\
    \partial_t H + \nabla\cdot\left(\vel H\right)&=0&&\quad\text{ in }(0,T)\times\Omega,\label{eq:model:H}
\end{alignat}
\end{subequations}
where $\rho_{\text{ice}}$ are a given sea-ice density, $f_c$ is the
Coriolis parameter and $\bs e_r$ is the unit vector normal to the
surface.
The internal stresses $\boldvar\sigma=\boldvar\sigma(\vel, A, H)$ in
the sea-ice are modeled by the VP rheology, \cite{Hibler1979}
\begin{equation}
    \boldvar\sigma := 2\eta\strainratetensor'(\vel) + \zeta \text{tr}(\strainratetensor(\vel)) I - \frac{P}{2} I
\end{equation}
with viscosities $\eta,\,\zeta$ given by
\begin{equation*}
\eta = e^{-2}\zeta,\quad \zeta = \frac{P}{2\Delta(\vel)}, \quad \text{where }
\end{equation*}
\begin{equation}\label{eq:Delta}
\Delta(\vel) :=\sqrt{2e^{-2}\strainratetensor'(\vel):\strainratetensor'(\vel) + \text{tr}(\strainratetensor(\vel))^2 + \Delta_{\text{min}}^2},
\end{equation}
$\strainratetensor(\vel):=\frac
12(\nabla \vel+\nabla \vel^T)\in \mathbb R^{2\times 2}$ is the strain rate
tensor, $\strainratetensor'(\vel):=\strainratetensor(\vel) - \frac 12 \text{tr}(\strainratetensor(\vel))\bs I \in \mathbb R^{2\times 2}$ is the deviatoric strain rate tensor and ``:'' denotes the Frobenius inner product between tensors.  Moreover, $e>0$ is the ratio of the main axes of the elliptic yield curve (typically, $e=2$).
Following \cite{KreyscherHarderLemkeEtAl00}, we use a smooth transition between
plastic and viscous states using a small constant
$\Delta_{\text{min}}>0$ in \eqref{eq:Delta}.
The spatially varying ice strength $P=P(\bs x)$ is modeled as
\begin{equation}\label{eq:ice-strength}
    P = P^*H\exp(-C(1-A))
\end{equation}
with constants $P^*$ and $C$.
The forcing from atmosphere and ocean are given by
\begin{equation}\label{eq:force}
    \bs\tau_\text{atm}(t) = C_a\rho_a \|\vel_{\text{a}}(t)\|_2 \vel_{\text{a}}(t) \quad\text{and}\quad 
    \bs\tau_\text{ocean}(\vel, t) = C_o\rho_o \|\vel_o(t) - \vel \|_2 (\vel_o(t)-\vel),
\end{equation}
respectively, where $\|\cdot\|_2$ denotes the Euclidean norm, $C_a$
and $C_o$ are the atmospheric and ocean drag coefficients, $\rho_a$
and $\rho_o$ the densities, and $\vel_a$ and $\vel_o$ the velocity
fields of near surface atmospheric and oceanic flows. 

To complete the formulation of the sea-ice system, we use the following initial and boundary conditions: 
\begin{alignat*}{2}
    \vel&=\vel_0&&\quad\text{ on } \{t=0\}\times \Omega,\\
    A = A_0,\quad H &= H_0&&\quad\text{ on } \{t=0\}\times \Omega,\\
    \vel &= \boldvar 0&&\quad\text{ on }I_t\times \partial\Omega,\\
    A = A^{\text{in}},\quad H &= H^{\text{in}} &&\quad\text{ on } I_t\times\Gamma^{\text{in}},
\end{alignat*}
where $\Gamma_{\text{in}}:=\{\boldvar x\in
\partial\Omega~\vert~\boldvar n\cdot \vel < 0\}$ denotes the part of
the boundary $\partial \Omega$ with incoming
characteristics. Recently, analytical results these equations have been reported
\cite{LiuThomasTiti21, BrandtDisserHallerEtAl21}. However, we focus on
the classical system equations \eqref{eq:model} and study solvers for
the time-discretized equations.
\section{Time discretization and implicit momentum equation time step}
To discretize \eqref{eq:model} in time, we use an explicit method for
the hyperbolic equations \eqref{eq:model:A} and \eqref{eq:model:H},
and an implicit method for the momentum equation
\eqref{eq:momen}. That is, given the triple $(\vel^n, A^n, H^n)$ at
time $t_n$, these variables are advanced to time $t_{n+1}:=t_n+\delta
t$ as follows. First, an explicit time step (e.g., an explicit Euler
step) is performed to compute the concentration and height variables
$A^{n+1}$ and $H^{n+1}$, respectively. Due to the stiffness of the
momentum equation \eqref{eq:momen}, using an explicit method for
computing the velocity $\vel^{n+1}$  requires extremely small time
steps \cite{Hibler1991}. Thus, here we use an implicit time stepping
method, namely the implicit Euler method.\footnote{As not uncommon in sea
ice models, the (typically small) Coriolis force term is treated
explicitly. This results in a symmetric system matrix.} Other implicit time stepping methods lead to similar
nonlinear problems to be solved in each time step.

Next, we show that the result $\vel^{n+1}$ of this implicit step can also be
found as minimizer of an appropriately chosen convex energy minimization
problem.  The formulation as an energy minimization problem builds on
\cite{MehlmannRichter2016newton, RichterMehlmann2018, Mehlmann2019},
where it is shown that
the stress tensor can be represented in symmetric form and the
derivative of the momentum equation is positive definite, indicating
the underlying energy minimization problem. We consider the
functional
$\Phi: V \to \mathbb R$
defined over a function space $V\supset H^1_0(\Omega)^2$.
\begin{equation}\label{eq:energy}
\begin{aligned}
     \Phi (\vel) :=& \int_{\Omega} \frac{1}{2}\rho_{\text{ice}} H^{n+1} \|\vel\|_2^2 d\bs x - \int_\Omega \rho_{\text{ice}} H^{n+1} \vel^{n}\cdot \vel d\bs x\\
    &+ \delta t \left(\int_\Omega \frac{P^{n+1}}{2} \big[\Delta(\vel) - \text{tr}(\strainratetensor(\vel))\big] d\bs x +  \int_{\Omega}\frac{1}{3} C_o\rho_o\|\vel_o - \vel\|_2^3d\bs x\right)\\
    &
    + \delta t\left(\int_\Omega \rho_{\text{ice}}H^{n+1} f_c\bs e_r\times (\vel^n-\vel_o)\cdot \vel d\bs x -\int_\Omega \bs\tau_\text{atm}(t_{n+1})\cdot \vel d\bs x\right).
\end{aligned}
\end{equation}
Note that the space $V$ over which $\Phi$ is well-defined is slightly
larger than $H^1_0(\Omega)^2$ as a finite value of $\Phi(\vel)$ does
not require that $\vel$ has square-integrable derivatives.
Identifying the appropriate space $V$ for the semi-discretized problem
rigorously and proving solution existence/uniqueness is beyond the
scope of this paper. We refer to comments in \cite{Mehlmann2019}
and for a function space analysis of the time-continuous problem to
\cite{BrandtDisserHallerEtAl21, LiuThomasTiti21}. However, in the next
theorem we show that $\Phi$ is convex and bounded from below, and that
the implicit time step equations can be derived as first-order
necessary conditions for a minimizer of $\Phi$. For simplicity, we
restrict ourselves to $\vel \in H_0^1(\Omega)^2$.
\begin{theorem}
  The functional $\Phi$ defined in \eqref{eq:energy} is convex and bounded from below.
  Moreover, the solution at next time step
  $\vel^{n+1}\in V$ satisfies the stationary condition
  $\Phi'(\vel^{n+1})(\bs \phi)=0$ for all $\bs \phi\in V$.
\end{theorem}
\begin{proof}
The functional $\Phi$ is convex in the velocity $\vel$ and
in the strainrate tensor $\strainratetensor(\vel)$. This follows since
the terms in \eqref{eq:energy} are either linear in $\vel$, or
quadratic or cubic in $\|\vel\|$, and since $\Delta(\vel)$ is convex
in $\strainratetensor(\vel)$, which itself is a convex function of
$\vel$.  Moreover, the coefficients of the quadratic and
cubic terms are pointwise non-negative. 
The functional is bounded from below as terms in which $\vel$ appears linearly are dominated by quadratic or third-order terms in $\vel$. 
Moreover,
$\Delta(\vel) - \text{tr}(\strainratetensor(\vel)\ge 0$ for all $x\in
\Omega$ by definition of $\Delta(\vel)$.

Let us now compute the first variations of $\Phi(\vel)$, i.e.,
$\Phi'(\vel)(\bs\phi)$ following basic methods from variational calculus.
For this purpose, we use the identities
\begin{equation}\label{eq:minPhi}
\begin{aligned}
\left(\frac{1}{2}\rho_{\text{ice}} H^{n+1} \|\vel\|_2^2\right)'(\vel)(\bs \phi) &= \rho_{\text{ice}} H^{n+1} \vel\cdot \bs \phi,\\
\left(\frac{1}{3} C_o\rho_o\|\vel_o - \vel\|_2^3\right)'(\vel)(\bs\phi)&=
C_o\rho_o\|\vel_o - \vel\|_2^2 \left(\|\vel_o - \vel\|_2\right)'(\vel)(\bs\phi)= C_o\rho_o\|\vel_o - \vel\|_2\left(\vel_o - \vel\right)\cdot \bs\phi,\\
\left(\frac{P^{n+1}}{2} \Delta(\vel)\right)'(\vel)(\bs \phi) &=
\frac{P^{n+1}}{4\Delta(\vel)}\Big(2e^{-2}\strainratetensor'(\vel):\strainratetensor'(\vel) + \text{tr}(\strainratetensor(\vel))^2 + \Delta_{\text{min}}^2\Big)'(\vel)(\bs\phi)\\
&=
\frac{P^{n+1}}{4\Delta(\vel)}4e^{-2}\strainratetensor'(\vel):\strainratetensor'(\bs\phi) + \frac{P^{n+1}}{2\Delta(\vel)}\text{tr}(\strainratetensor(\vel))\text{tr}(\strainratetensor(\bs\phi))\\
&=2\eta\strainratetensor'(\vel):\strainratetensor'(\bs\phi) + \zeta\text{tr}(\strainratetensor(\vel))\text{tr}(\strainratetensor(\bs\phi))\\
&=\left(2\eta\strainratetensor'(\vel) + \zeta\text{tr}(\strainratetensor(\vel))I\right):\nabla\bs\phi.\\
\end{aligned}
\end{equation} 
%
Note that above we used that $\strainratetensor':\bs I=0$.
Thus, $\Phi'(\vel)(\bs\phi)=0$ is equivalent
to
\begin{subequations}\label{eq:t-step}
\begin{equation}\label{eq:Av=F}
\mathcal{A}(\vel^{n+1}, \bs\phi) = F(\bs\phi),
\end{equation}
where $\mathcal{A}(\cdot\,, \cdot)$ and $F(\cdot)$ are defined as
\begin{align}
\mathcal{A}(\vel^{n+1}, \bs \phi) &:= (\rho_{\text{ice}}H^{n+1} \vel^{n+1}, \bs \phi)+\delta t (\bs \sigma(\vel^{n+1}, A^{n+1}, H^{n+1}), \nabla \bs \phi) - 
\delta t (\bs\tau_\text{ocean}(\vel^{n+1})), 
\label{eq:A}\\
F(\bs \phi)&:= (\rho_{\text{ice}}H^{n+1} \vel^{n}, \bs \phi) -
\delta t (\rho_{\text{ice}}H^{n+1} f_c\bs e_r\times (\vel^n-\vel_o), \bs \phi)+ \delta t (\bs\tau_\text{atm}(t_{n+1}), \bs \phi).\label{eq:F}
\end{align}
\end{subequations}
Here, $(\cdot\,,\cdot)$ denotes the standard $L_2$-inner product for vectors or matrices. Thus, \eqref{eq:Av=F}--\eqref{eq:F} is also the weak form of an implicit time step for the momentum
equation \eqref{eq:momen}.
\end{proof}

The above results shows that the new time level $\vel^{n+1}$ can be
found by solving
\begin{equation}\label{eq:Phi-opt}
  \min_{\vel \in V} \Phi (\vel)
\end{equation}
with an appropriately defined function space $V$. Such a variational
calculus perspective may be useful
to establish rigorous existence and uniqueness results for
\eqref{eq:Av=F}. 
If $V$ is approximated by a finite-dimensional
space $V_h$, e.g., a space of piecewise polynomials as common in
finite element methods, then it is straightforward to show that the (now
fully discrete) time step $\vel^{n+1}_h$ is the solution to a
finite-dimensional convex optimization problem.

The optimization formulation \eqref{eq:Phi-opt} also shows that the
linearization of \eqref{eq:Av=F}, which is the second variation of
$\Phi$, must be symmetric. This is also discussed in
\cite{MehlmannRichter2016newton}.  Additionally, the objective
$\Phi(\cdot)$ can also be used in an iterative method to verify that
progress has been made towards the computation of $\vel^{n+1}$.
In particular, we use decay with respect to $\Phi$ in a line search
procedure in our Newton methods instead of the more commonly used
nonlinear residual norm.

Note that solving the momentum equation \eqref{eq:t-step} (or its
spatial discretizations) remains
challenging due to the nonlinearity reflecting the VP
constitutive relation. The main focus of this paper is on Newton-type
methods to solve this equation robustly and efficiently. In
particular, we will propose a novel linearization method for
\eqref{eq:Av=F} in the next section.

\section{Newton methods for the momentum equation}\label{sec:newton}
For convenience of the notation, we introduce some notation 
following \cite[Eq.(5.13)]{Mehlmann2019}. Namely, we define the tensor
\begin{equation*}
\auxtau(\vel):=e^{-1}\strainratetensor' + \frac
12\text{tr}(\strainratetensor)\bs I,
\end{equation*}
which allows us to write, since $\strainratetensor':\bs I=0$, that
\begin{equation*}\label{eq:sigma-delta}
  \Delta(\vel) =
  \sqrt{\Delta_{\text{min}}^2 + 2\auxtau(\vel): \auxtau(\vel)}.
\end{equation*}
Following \cite{Mehlmann2019}, 
the term involving $\bs \sigma = \bs\sigma(\vel^{n+1}, A^{n+1}, H^{n+1})$ in \eqref{eq:A} can be written as
\begin{equation}\label{eq:sigma}
    (\bs \sigma, \nabla \bs\phi) 
  = \left(\frac{P^{n+1}}{\Delta(\vel)}
  \auxtau(\vel), \auxtau(\bs\phi)\right) - 
\left(\frac{P^{n+1}}{2}, \text{tr}(\strainratetensor(\bs\phi))\right),
\end{equation}
where $P^{n+1}$ is the ice strength \eqref{eq:ice-strength} evaluated
at $H^{n+1}$ and $A^{n+1}$ and $(\cdot\,,\cdot)$ is the $L^2$-inner product of two matrix functions.

\subsection{Standard Newton method}
A Newton step to solve \eqref{eq:Av=F} is of the following form:
Compute the Newton update $\tilde{\vel}\in V$ such that
\begin{align}\label{eq:Newton-primal}
	\mathcal{A}'(\vel_{l})(\tilde{\vel},\bs \phi) &= F(\bs \phi) -  A(\vel_{l},\bs \phi) \quad \text{for all } \bs \phi\in V,
\end{align}
and then perform the Newton update step $\vel_{l+1} :=\vel_{l} + \alpha \tilde{\vel}$ with a step length $\alpha\le 1$. Here, $l$ denotes the index for the Newton iteration to compute the velocity $\vel^{n+1}$.
The Jacobian on the left hand side in \eqref{eq:Newton-primal} is given by
\begin{subequations}\label{eq:standardNewton}
\begin{equation}\label{eq:standardNewton1}
\begin{aligned}
	\mathcal{A}'(\bs v_{l})(\tilde{\vel},\bs \phi)
= (\rho_{\text{ice}}H^{n+1} \tilde{\vel}, \bs \phi)& \\[1ex] 
    + \delta t 
    \left(\frac{P^{n+1}}{\Delta(\vel_l)}\auxtau(\tilde{\vel}), \auxtau(\bs\phi)\right)&
    -\delta t\left(2P^{n+1}\frac{\auxtau(\vel_l)\otimes\auxtau(\vel_l)}{\Delta(\vel_l)^3}\auxtau(\tilde{\vel}), \auxtau(\bs\phi)\right)
    - \delta t \bs\tau_\text{ocean}'(\bs v_{l})(\tilde{\vel},\bs \phi),
\end{aligned}
\end{equation}
where 
\begin{equation}\label{eq:standardNewton2}
\bs\tau_\text{ocean}'(\bs v)(\tilde{\vel},\bs \phi)=
\left(-\rho_o C_o \|\vel_o - \vel\|_2 \tilde{\vel}, \bs \phi \right) +
\left(-\rho_o C_o \frac{(\vel_o - \vel)^T\tilde{\vel}}{\|\vel_o - \vel\|_2} (\vel_o - \vel), \bs \phi \right).
\end{equation}
\end{subequations}
Here, $\otimes$ in \eqref{eq:standardNewton1} denotes the outer product between two matrices, whose result is a 4th-order tensor, and we have used the identity $(\bs a:\bs b)\bs c = (\bs a\otimes \bs c)\bs b$ for matrices $\bs a, \bs b, \bs c$. 
Hence, the term involving $\otimes$ in \eqref{eq:standardNewton1} equals 
$\delta t\left({2P^{n+1}}{(\Delta(\vel_l))^{-3}}(\auxtau(\vel_l):\auxtau(\tilde{\vel}))\auxtau(\vel_l),\auxtau(\bs\phi)\right)
=\delta t\left( {2P^{n+1}}{(\Delta(\vel_l))^{-3}}(\auxtau(\vel_l):\auxtau(\tilde{\vel})),\auxtau(\vel_l):\auxtau(\bs\phi)\right)$,
which is also how it is implemented based on weak forms.
To summarize, each Newton iteration amounts to solving a linear system with the operator as defined in \eqref{eq:standardNewton1}. 
\subsection{A stress--velocity Newton method}\label{sec:svnewton}
Next, we present an alternative formulation for the nonlinear equation \eqref{eq:Av=F}. The main difference of this formulation is that it uses a reformulation of the VP nonlinearity that is 
better suited for Newton linearization.
To arrive at this alternative formulation, we introduce the independent variable
\begin{equation}\label{eq:pi}
\bs \pi := \frac{\auxtau(\bs v)}{\sqrt{\Delta_{\min}^2 +
  2\auxtau(\bs v):\auxtau(\bs v)}} \in \mathbb R^{2\times 2}.
\end{equation}
Since the denominator is always positive, the definition is equivalent to the nonlinear equation
\begin{equation}\label{eq:NCPreformulation}
\bs r(\bs \pi,\bs v):= \bs \pi \sqrt{\Delta_{\min}^2 +
  2\auxtau(\bs v):\auxtau(\bs v)} - \auxtau(\bs v) = 0.
\end{equation}
With this auxiliary variable $\bs \pi$, \eqref{eq:sigma} simplifies to
\begin{equation}\label{eq:sigma-pi}
\left(P^{n+1}\bs \pi,\auxtau(\bs
  \phi)\right) - \left(\frac{P^{n+1}}{2},\text{tr}(\strainratetensor(\bs\phi))\right) := B(\vel, \bs\pi, \bs\phi).
\end{equation}
Substituting  \eqref{eq:sigma-pi} for \eqref{eq:sigma} in \eqref{eq:Av=F} and adding \eqref{eq:res=0} as a new equation, we obtain a new formulation of the momentum equation \eqref{eq:Av=F} for unknowns 
$(\vel, \bs\pi)$
\begin{subequations}
\begin{align}
 (\rho_{\text{ice}}H^{n+1} \vel, \bs \phi) + 
    \delta t B(\vel,\bs\pi,\bs\phi) - 
    \delta t (\bs\tau_\text{ocean}(\vel)) &= F(\bs\phi)\label{eq:newAv=f}\\
    \bs r(\bs \pi,\bs v) &= 0\label{eq:res=0}.
\end{align}
\end{subequations}
The corresponding Newton method: Find a Newton update
$\tilde{\vel}$ and $\tilde{\bs \pi}$ such that 
\begin{subequations}
\begin{align}
    (\rho_{\text{ice}}H^{n+1} \tilde{\vel}, \bs \phi)
    + \delta t B'(\vel_l,\bs\pi_l)(\tilde{\vel},\tilde{\bs \pi},\bs\phi)
    - \delta t \bs\tau_\text{ocean}'(\bs v_{l})(\tilde{\vel},\bs \phi) &=\notag\\
    \quad\quad\quad\quad\quad\quad F(\bs\phi) 
    - (\rho_{\text{ice}}H^{n+1} \vel_l, \bs \phi)
    - \delta t B(\vel_l,\bs\pi_l,\bs\phi) &
    + \delta t \bs\tau_\text{ocean}(\vel_l)\\
    \bs r'(\bs \pi_l,\bs v_l)(\tilde{\bs \pi},\tilde{\vel}) &= -\bs r(\bs \pi_l,\bs v_l)\label{eq:sv-r}
\end{align}
\end{subequations}
for all $\bs\phi\in V$, and then perform the Newton update step  
$ \bs v_{l+1} :=\bs v_{l} + \alpha \tilde{\vel}$ and
$ \bs \pi_{l+1} :=\bs \pi_{l} + \alpha \tilde{\bs \pi}$.
with a step length $\alpha\le 1$.

The Gateaux derivative $\bs r'(\bs \pi,\bs v)$ with respect to $\vel$ and $\bs\pi$ can be computed as follows:
\begin{equation}\label{eq:R'}
    \bs r'(\bs \pi,\bs v)(\tilde{\bs\pi}, \tilde{\vel}) 
    = \tilde{\bs \pi}\Delta(\bs v) + \frac{2\auxtau(\tilde{\vel}):\auxtau(\vel)}{\Delta(\bs v)}\bs\pi - \auxtau(\tilde{\vel}) 
    = \tilde{\bs \pi}\Delta(\bs v) + \frac{2\auxtau(\bs v)\otimes \bs \pi}{\Delta(\bs v)}\auxtau(\tilde{\vel}) - \auxtau(\tilde{\vel}).
\end{equation}
Thus, the Newton step \eqref{eq:sv-r} is
\begin{equation}\label{eq:Newton-r}
\tilde{\bs \pi}\Delta(\bs v_l) + \frac{2\auxtau(\vel_l)\otimes \bs \pi_l}{\Delta(\bs v_l)}\auxtau(\tilde{\vel}) - \auxtau(\tilde{\vel}) = -\bs \pi_l \Delta(\bs v_l) + \auxtau(\bs v_l).
\end{equation}
Computing the remaining derivative terms in \eqref{eq:R'} is straightforward. Namely,  
$ B'(\vel, \bs\pi)(\tilde{\vel}, \tilde{\bs \pi}, \bs \phi) = \left(P^{n+1}\tilde{\bs \pi}, \auxtau(\bs\phi)\right)$
and $\bs \tau_\text{ocean}'$ is computed in \eqref{eq:standardNewton2}. 

Next, as in \cite{RudiShihStadler20}, we eliminate \eqref{eq:sv-r} from the system. 
Dividing the Newton step \eqref{eq:Newton-r} by $\Delta(\bs v_l)$ shows that
\begin{equation}\label{eq:tautilde}
\tilde{\bs \pi} = 
-\bs \pi_l 
+\frac{\auxtau(\bs v_l)}{\Delta(\bs v_l)}
+ \frac{\auxtau(\tilde{\vel})}{\Delta(\bs v_l)}
- \frac{2\auxtau(\vel_l)\otimes \bs \pi_l}{\Delta(\bs v_l)^2}\auxtau(\tilde{\vel})
.
\end{equation}
Using this in \eqref{eq:R'}, we find the following linear equation for $\tilde{\vel}$: 
\begin{equation}
\begin{aligned}\label{eq:linearSVNewton}
  (\rho_{\text{ice}}H^{n+1} \tilde{\vel}, \bs \phi) 
  + \delta t \left(\frac{P^{n+1}}{\Delta(\bs v_l)}\auxtau(\tilde{\vel}),\auxtau(\bs\phi) \right)
  - \delta t \left(2P^{n+1}\frac{\auxtau(\vel_l)\otimes \bs \pi_l}{\Delta(\bs v_l)^2}\auxtau(\tilde{\vel}), \auxtau(\bs \phi) \right) 
  - \delta t \bs\tau_\text{ocean}'(\bs v_{l})(\tilde{\vel},\bs \phi) \\
  = F(\bs\phi)
  - (\rho_{\text{ice}}H^{n+1} \vel_l, \bs \phi)
  - \delta t \left(\frac{P^{n+1}}{\Delta(\bs v_l)}\auxtau(\bs v_l),\auxtau(\bs\phi)\right)
  + \delta t \left(\frac{P^{n+1}}{2},\text{tr}(\strainratetensor(\bs\phi))\right)
  + \delta t \bs\tau_\text{ocean}(\vel_l).
\end{aligned}
\end{equation}
Note that the right hand side in the above system coincides with the right hand 
side in the original Newton method \eqref{eq:standardNewton1}. 
The left hand side is similar to the one in \eqref{eq:standardNewton1}, with the difference that
$\auxtau(\bs v)/\Delta(\bs v_l)$ in the second term of \eqref{eq:standardNewton1} is replaced by $\bs
\pi_l$ in \eqref{eq:linearSVNewton}. Having solved \eqref{eq:linearSVNewton} for 
the Newton step $\tilde{\vel}$, we use \eqref{eq:tautilde} to
compute the Newton step for $\bs\pi$, which involves a matrix vector
multiplication and does not require solving a linear system.

On top of the above derivation, analogously as in \cite{RudiShihStadler20}, we enforce that the 4th-order tensor in \eqref{eq:linearSVNewton} is symmetric and that $\sqrt{2}\bs \pi_l$
remains in the unit sphere by replacing
\begin{equation}\label{eq:modification}
\auxtau(\bs v_l)\otimes\bs \pi_l \quad \text{by} \quad \frac{\auxtau(\vel_l)\otimes\bs \pi_l + \bs \pi_l \otimes\auxtau(\vel_l)}{2\max(1,\sqrt{2\bs \pi_l: \bs \pi_l})}.
\end{equation}
The same modifications are applied to the last term in \eqref{eq:tautilde} for consistency.
We note that both modifications vanish upon convergence. First, upon convergence, 
i.e., $\bs\pi$ satisfies \eqref{eq:pi}, $\bs \pi$ is a multiple of $\bs\tau$ and thus symmetric
so the symmetrization has no or very little effect when close to the solution. 
Second, upon convergence, $\sqrt{\bs \pi_l: \bs \pi_l}\le 1/\sqrt{2}$ and thus the scaling has no effect.

\section{Numerical Results}\label{sec:num}
In this section, we compare the convergence behavior of various
Newton-type methods and implementations for two test problems. The
first problem (\Cref{sec:testcase1}) is a challenging single
time step problem for the momentum equation. The second problem
(\Cref{sec:testcase2}) is a benchmark from \cite{Mehlmann2021}, which
includes the evolution of sea-ice concentration $A$ and
sea-ice thickness $H$. Both problems use the physical parameters
summarized in \Cref{tab:params}. In particular,
we assume that there is no Coriolis force and thus $f_c=0$.
Finally, in \Cref{sec:linear}, we
propose and study iterative solvers for the linearized systems arising
in each Newton step.
\subsection{Methods and implementations}\label{sec:methods}
We compare the performance of various Newton-type methods, including
the proposed stress--velocity Newton method (denoted as ``\svnewton{}''), 
a standard Newton method (denoted as ``\stdnewton{}'') ,
a modified Newton method from \cite{MehlmannRichter2016newton} 
implemented in the academic software library \textsc{Gascoigne 3D}
\cite{Gascoigne} (denoted as ``\gascoigne{}''), and a Jacobian-free Newton-Krylov (JFNK) solver
\cite{Lemieux2010,Losch2014} from the \textsc{MITgcm} library
\cite{Adcroft2018} (denoted as ``\mitgcm{}''). For details about these solvers, we refer to the
literature above. Our comparisons with \gascoigne{} and
\mitgcm{} are only in terms of the Newton iterations and thus
implementation and discretizations aspects play a limited role.
 All implementations use the same smooth transition
between the viscous and plastic regime specified in \eqref{eq:Delta}.
In our own implementation, we additionally compare the performance of
various solvers for the linearized systems and report timings and
parallel scalability results in \Cref{sec:linear}.

\begin{table}
    \centering
    \begin{tabular}{lll}
    \toprule
    Parameter & Definition & Value  \\\midrule
    $\rho_{\text{ice}}$ & sea-ice density & $\SI{900}{\kg/\m^3}$\\
    $\rho_{\text{a}}$ & air density& $\SI{1.3}{\kg/\m^3}$\\
    $\rho_{\text{o}}$ & water density& $\SI{1026}{\kg/\m^3}$\\
    $C_{\text{a}}$ & air drag coefficient & $1.2\times 10^{-3}$\\
    $C_{\text{o}}$ & water drag coefficient& $5.5\times 10^{-3}$\\
    $P^*$ & ice strength parameter & $\SI{27.5}{\N/\m^2}$\\
    $C$ & ice concentration parameter& 20\\
    $e$ & ellipse ratio & 2\\
    $f_c$ & Coriolis parameter & 0\\
    \bottomrule
    \end{tabular}
    \caption{Physics parameters used in the momentum equation for both test problems.}
    \label{tab:params}
\end{table}

Our implementations of the standard Newton method and the
stress--velocity Newton methods\footnote{Source codes are available at
\url{https://github.com/MelodyShih/SeaIce_svNewton}.} are based on the
open source finite element library Firedrake
\cite{Dalcin2011,Rathgeber2016,Acosta2011,Chaco95,Mitchell2016} and
its interface with the parallel linear algebra toolkit \textsc{PETSc}
\cite{petsc-user-ref,petsc-efficient}.  We use structured
quadrilateral meshes to discretize the spatial domain $\Omega$.
Unless otherwise specified, we use linear $\mathbb{Q}_1$ elements
to discretize the sea-ice velocity $\vel$.  Finite volumes are used
for the sea-ice concentration $A$ and thickness $H$, i.e.,
$\mathbb{Q}_0^{\text{disc}}$ elements. Our implementation (and the
proposed method) also allows
to use higher-order discretizations, and we show results obtained with
quadratic finite elements in space for the problem discussed in \Cref{sec:testcase2}.
%
We terminate the nonlinear iteration when we obtain
$10^4$ reduction in the $L_2$-norm of the nonlinear residual
\ $\|A(\vel_l,\bs \phi) - F(\bs \phi)\|_2$.  To compute a step length
for the Newton updates $\tilde{\vel}$ (and $\tilde{\bs \pi}$), we use
backtracking with the optimization objective \eqref{eq:energy},
starting from unit step length and halving the step length if the
objective does not decrease. Using the norm of the nonlinear residual
rather than the optimization objective for backtracking results in
similar results.  We either use the parallel direct
sparse solver MUMPS \cite{MUMPS01} to solve the linearized problems
arising in each Newton step, or an iterative parallel Krylov solver,
which we detail in \Cref{sec:linear}.

\subsection{Problem I: Single momentum equation solve}\label{sec:testcase1}
This is a challenging single momentum solve problem due to the severe
nonlinearity of the constitutive relation.
We consider the square domain $\Omega=(0, \SI{512}{\km})^2$. 
The atmospheric and ocean forcing are constant in space and time and are
given by
$\vel_o(\bs x)=\boldvar 0~\SI{}{\m\s^{-1}}$ and $\vel_a(\bs x)= (5,5)~\SI{}{\m\s^{-1}}$, 
and we consider a sea-ice concentration $A$ and the thickness $H$ given by
\begin{align*}
A(\bs x) :=1-0.5e^{-800|r(\bs x)|} -0.4e^{-90 |r_1(\bs x)|} -0.4e^{-90 |r_1(\bs x)+0.7|}\quad\text{and}\quad
H(\bs x):=2 A(\bs x),
\end{align*}
where
\begin{align*}
  r(\bs x) = r(x,y)&=0.04 - \Big(\frac{x}{\SI{1000}{\km}}-0.25\Big)^2-\Big(\frac{y}{\SI{1000}{\km}}-0.25\Big)^2,\\
r_1(\bs x) = r_1(x,y)&=0.1 +\Big(\frac{2 x}{\SI{1000}{\km}}\Big )^2-\Big(\frac{2y}{\SI{1000}{\km}}\Big).
\end{align*}
These sea-ice concentration and thickness field have narrow regions
where they are reduced, as shown in the top left of \Cref{fig:cond}. The
initial velocity is $\vel(0,\bs x)=\boldvar 0~\SI{}{\m\s^{-1}}$ and we assume zero
Dirichlet boundary conditions for the velocity $\vel$.  We use a time
step size of $\delta t = \SI{0.5}{\hour}$ for the first (and only)
time step, and unless otherwise specified, a spatial mesh size of $\Delta x = \Delta y = \SI{1}{\km}$,
and $\Delta_\text{min}=2\times10^{-9}$. The sea-ice velocity solution
$\vel$ is shown in the top right of \Cref{fig:cond}.  Next, we compare
the performance of various solvers.

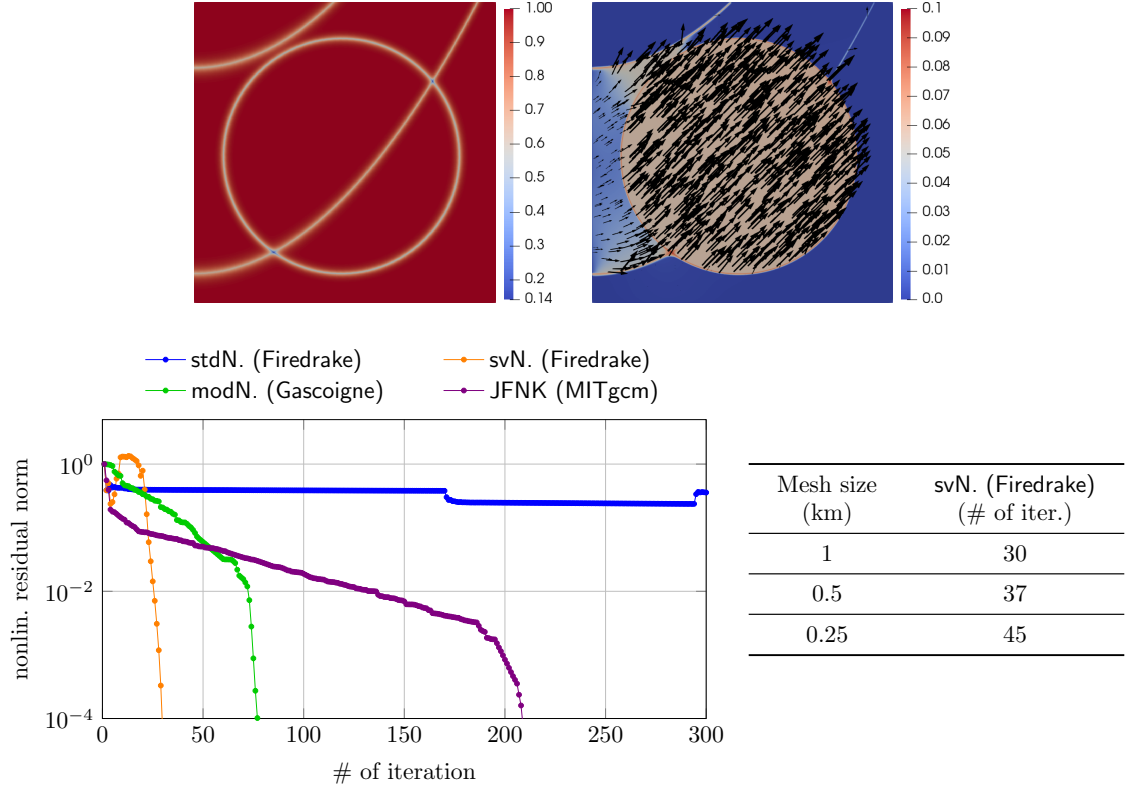
\begin{figure}
\centering
\includegraphics[align=c,height=4cm]{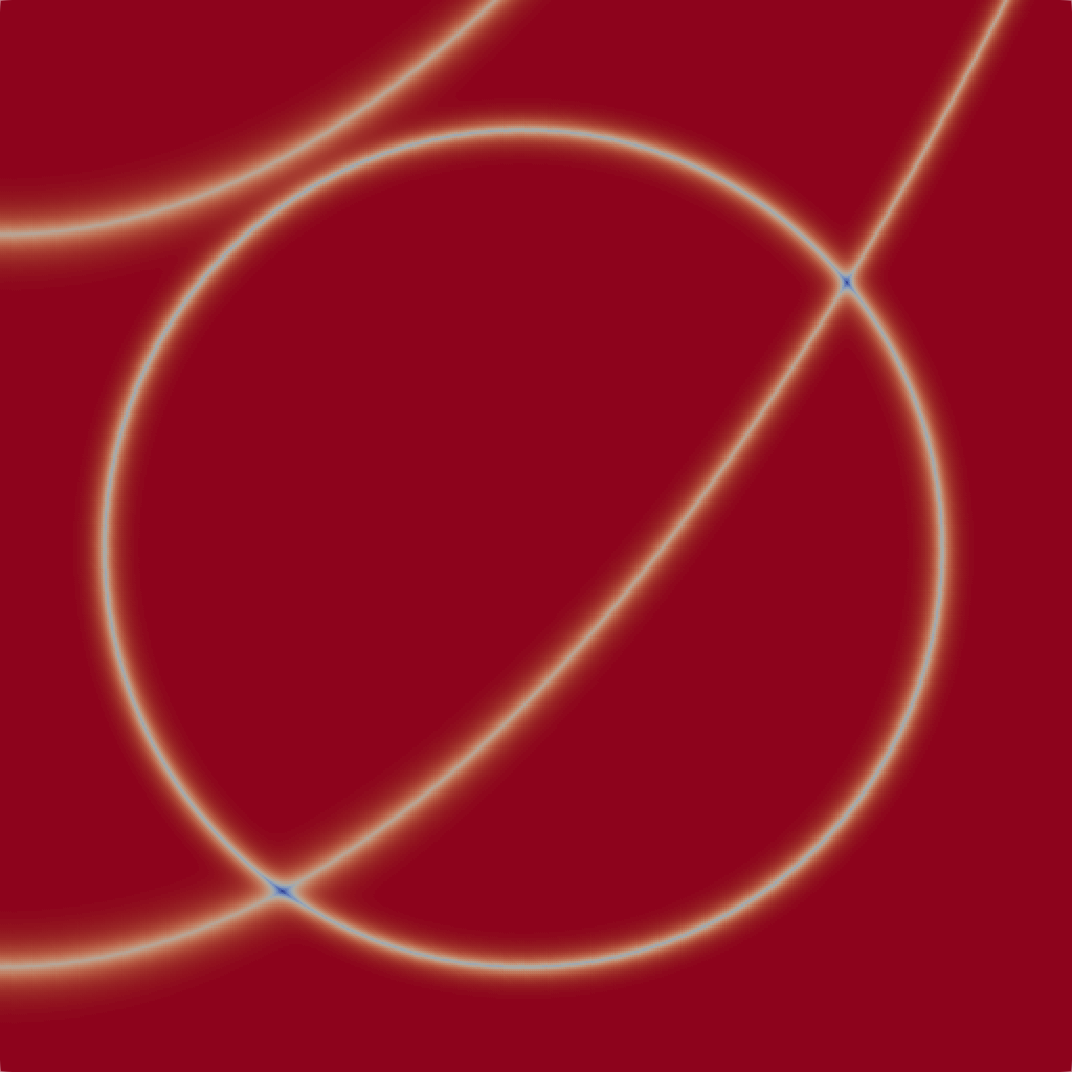}
\includegraphics[align=c,height=4cm]{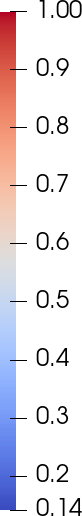}
\hspace{0.3cm}
\includegraphics[align=c,height=4cm]{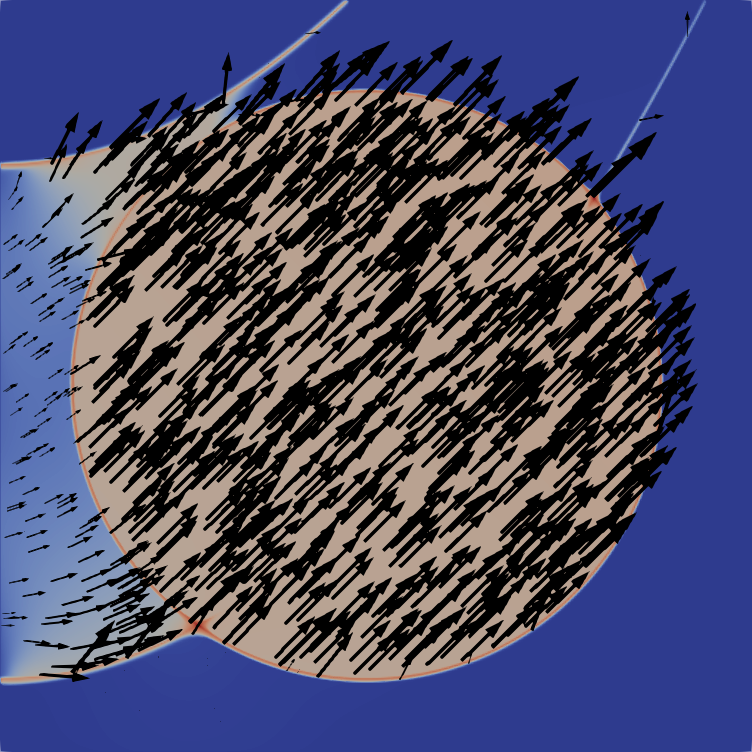}
\includegraphics[align=c,height=4cm]{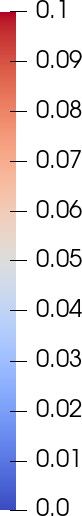}\\[3ex]
\centering \begin{tikzpicture}[baseline=(current axis.outer east),scale=0.9]
\centering
\begin{semilogyaxis}[
    height=6cm,
    width=10.5cm,
    title={},
    xlabel={\# of iteration},
    ylabel={nonlin.\ residual norm},
    xmin=0, xmax=300,
    ymax=5e0,
    ymin=1e-4,
    ytick={1e0, 1e-2, 1e-4},
    grid=major,
    /tikz/every even column/.append style={column sep=0.8cm},
    legend style={at={(0.5, 1.27)},
                  draw=none,
                  anchor=north},
    legend columns=2,
    legend cell align=left]
    \addplot[blue,mark=*, mark size=1pt] table[skip first n=1] {\data/newton_res_norm.out};
	\addlegendentry{\stdnewton{}}
    \addplot[orange,mark=*, mark size=1pt] table[skip first n=1] {\data/stressvel_res_norm.out};
	\addlegendentry{\svnewton{}}
    \addplot[green!80!black,mark=*, mark size=1pt] table [col sep=comma, y index=1] {\data/Newton_2km_stationaer.csv};
    \addlegendentry{\gascoigne}
    \addplot[violet,mark=*, mark size=1pt] table [col sep=comma, y index=1] {\data/test2_iterations_2km.csv};
    \addlegendentry{\mitgcm}
\end{semilogyaxis}
\end{tikzpicture}
\scalebox{0.9}{
\begin{tabular}{cc}
\toprule
\begin{tabular}{c}
Mesh size \\
 (km)
\end{tabular}
 & 
\begin{tabular}{c}
\svnewton{} \\
(\# of iter.)
\end{tabular}\\\midrule
1    & 30 \\\midrule
0.5  & 37 \\\midrule
0.25 & 45 \\
\bottomrule
\end{tabular}
}
\caption{Setup and convergence for Problem I. The top left image shows
  the concentration field $A$, which has narrow regions with reduced
  concentration. The ice height $H$ is a constant multiple of $A$.
  The top right figure shows the computed sea-ice velocity $\vel$.
  Arrows show the velocity direction and the background color depicts
  the velocity magnitude, which has sharp gradients as a result of the
  viscous-plastic constitutive relation.  Shown at the bottom left is the
  $L_2$-norm of the nonlinear residual in each iteration for various
  methods and implementations with mesh size $\Delta x = \SI{1}{km}$. 
  On the bottom right, we show the number of iterations of \svnewton{} for two finer mesh sizes $\Delta x = \SI{0.5}{km}, \SI{0.25}{km}$.}
\label{fig:cond}
\end{figure}

The bottom figure in \Cref{fig:cond} compares
the nonlinear convergence of various solvers.  We observe a significant
reduction in the number of Newton iterations of the proposed method
\svnewton{} compared to the standard Newton method \stdnewton{}. The
\stdnewton{}  makes little progress within 200 iterations,
while the \svnewton{} is able to reduce the
residual by a factor of $10^4$ in 30 iterations. In addition, the
\svnewton{} outperforms the other two methods: in
terms of the number of iterations, it converges $2.6$ times faster
than the \gascoigne{}, and
about $7$ times faster than the \mitgcm{}. 
As will be shown in the next section, the difference in Newton iterations becomes even larger with higher grid resolution.
In the table in \Cref{fig:cond}, we also show that the number of Newton iterations only increases moderately upon mesh refinement, increasing from 30 for $\Delta x=\SI{1}{\km}$ to 37 for $\Delta x=\SI{0.5}{\km}$ and to 45 for $\Delta x=\SI{0.25}{\km}$. 

\subsection{Problem II: Time-dependent sea-ice benchmark from \cite{Mehlmann2021}}\label{sec:testcase2}
In this test problem, we solve the full time
dependent equations \eqref{eq:model} for the benchmark problem from
\cite{Mehlmann2021}.  In particular, we use the following parameters
and forcing. We consider a square domain $\Omega = (0,
\SI{512}{\km})^2$ and the atmospheric velocity $\vel_a$ and the ocean
flow velocity $\vel_o$ in \eqref{eq:force} are
\begin{equation*}
\begin{aligned}
    &\vel_o(t,x,y) = \SI{0.01}{\m\s^{-1}} 
    \begin{pmatrix}
        -1 + 2y/ (\SI{512}{\km})\\
        \phantom{-}1 - 2x/ (\SI{512}{\km})
    \end{pmatrix},\\
    &\vel_{\text{a}}(t,x,y) = 
    \omega(x,y)\bar{\vel}_{\text{a}}^{\max}
    \begin{pmatrix}
        \phantom{-}\cos(\alpha) &\sin(\alpha) \\
        -\sin(\alpha) &\cos(\alpha)
    \end{pmatrix}
    \begin{pmatrix}
        x - m_x(t) \\
        y - m_y(t)
    \end{pmatrix},
\end{aligned}
\end{equation*}
where
\begin{equation*}
\begin{aligned}
    \bar{\vel}_{\text{a}}^{\max} = \bar{\vel}_{\text{a}}^{\max}(t) &= 15 \text{ ms}^{-1} 
    \begin{cases}
    -\tanh ((\phantom{1}4-t)(\phantom{-}4+t)/2)  & t\in [0,4]~\SI{}{\day}\\
    \phantom{-}\tanh ((12-t)(-4+t)/2) & t\in [4,8]~\SI{}{\day}\\
    \end{cases},\\
    \alpha = \alpha(t) &= \begin{cases}
    72^\circ\quad t\in [0,4]~\SI{}{\day}\\
    81^\circ\quad t\in [4,8]~\SI{}{\day}\\
    \end{cases},\\
    m_x(t) = m_y(t) &= 
    \begin{cases}
    256\phantom{.6}~\SI{}{\km} + \SI{51.2}{\km/\day} \cdot t&  t\in [0,4] ~\SI{}{\day}\\
    665.6~\SI{}{\km} - \SI{51.2}{\km/\day} \cdot t & t\in [4,8]~\SI{}{\day}\\
    \end{cases}\\
    \omega(x,y)=\omega(t,x,y) &= \frac{1}{50}\exp{\left(-\frac{\sqrt{(x - m_x(t))^2 + (y - m_y(t))^2}}{\SI{100}{\km}}\right)}.
\end{aligned}
\end{equation*}
As initial conditions, we use
\begin{equation*}
    \vel(0, \bs x) = \boldvar 0~\SI{}{\m\s^{-1}},
    \hspace{0.5em} A(0,\bs x) = 1,\hspace{0.5em}
    H(0,x,y) = \SI{0.3}{\m} + \SI{0.005}{\m} \left(\sin\left( \frac{60 x}{\SI{1000}{\km}} \right) + \sin\left( \frac{30y}{\SI{1000}{\km}}\right)\right).
\end{equation*}
We use the time step size $\delta t = \SI{0.5}{\hour}$ for all our
simulations and unless otherwise specified,
$\Delta_\text{min}=2\times10^{-9}$. The resulting velocity field,
sea-ice concentration and shear deformation are shown in
\Cref{fig:vel,fig:iceconcen,fig:shear}. 
Note that the simulations yield more
small-scale solution features on finer meshes.
The results obtained with the \gascoigne{}
implementation largely coincides with our solution on the
corresponding meshes. To illustrate that the method does not rely on a
specific discretization, we also show a result with quadratic elements
for sea-ice velocity $\vel$ on a \SI{2}{\km} mesh. This discretization
has the same number of velocity unknowns as the linear element
discretization for $\vel$ on the \SI{1}{\km} mesh, and the solutions
are similar.
Higher-order spatial discretization are unlikely to result in
substantially more accurate solutions as the fields typically have
large variations of gradients due to the constitutive relation.
The \mitgcm{} solutions are slightly different.
Similar differences between solutions obtained
with different discretization have also been observed in the benchmark
study \cite{Mehlmann2021}.

\begin{figure}
\centering
\begin{tabular}{clccccl}
	$H$ &&& \multicolumn{3}{c}{$\vel$}\\
	\includegraphics[align=c,height=3cm]{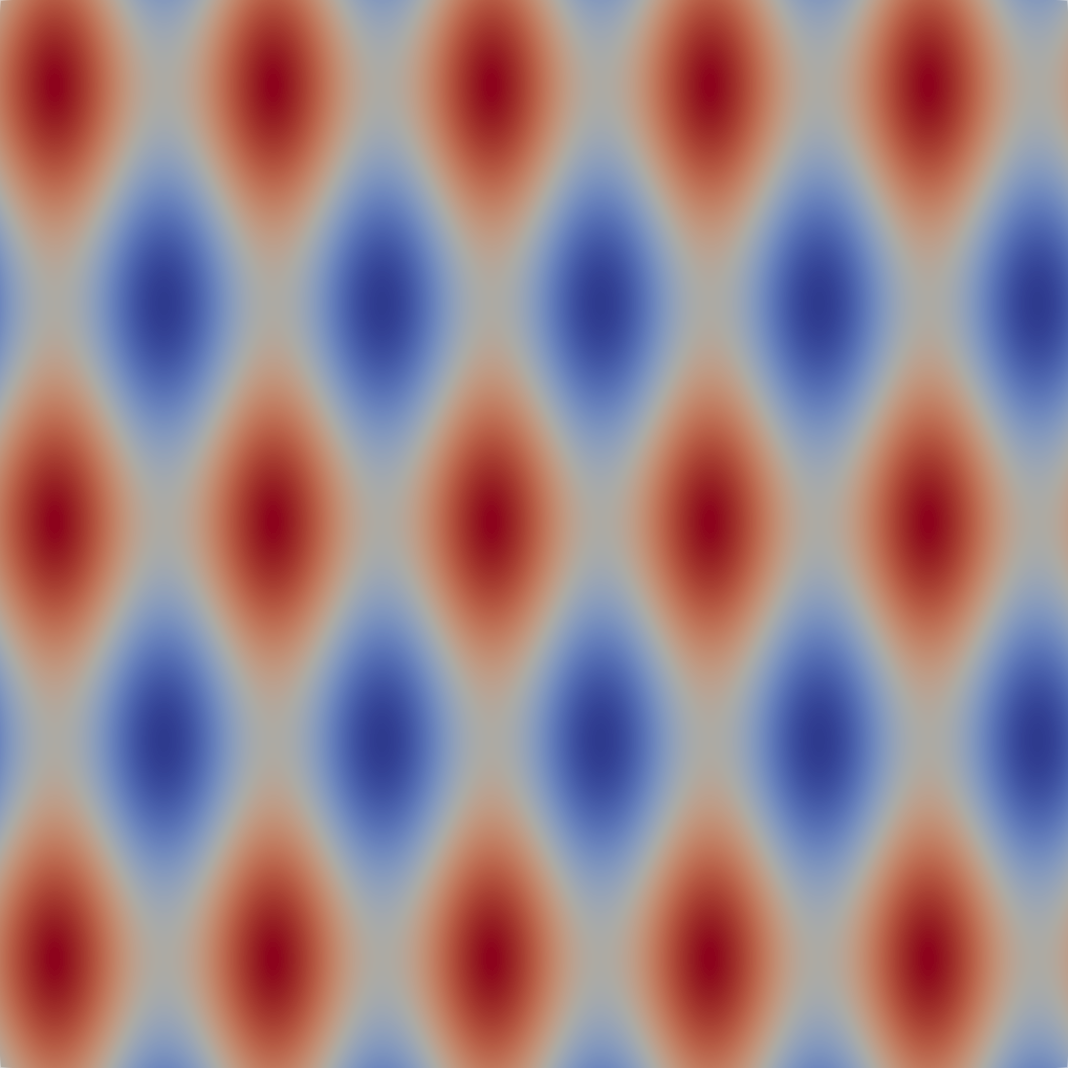} 
	&\multicolumn{1}{l}{\hspace{-0.2cm}\includegraphics[align=c,height=3cm]{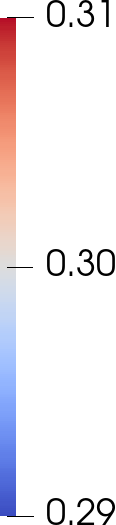}}
	&
    &\includegraphics[align=c,height=3cm]{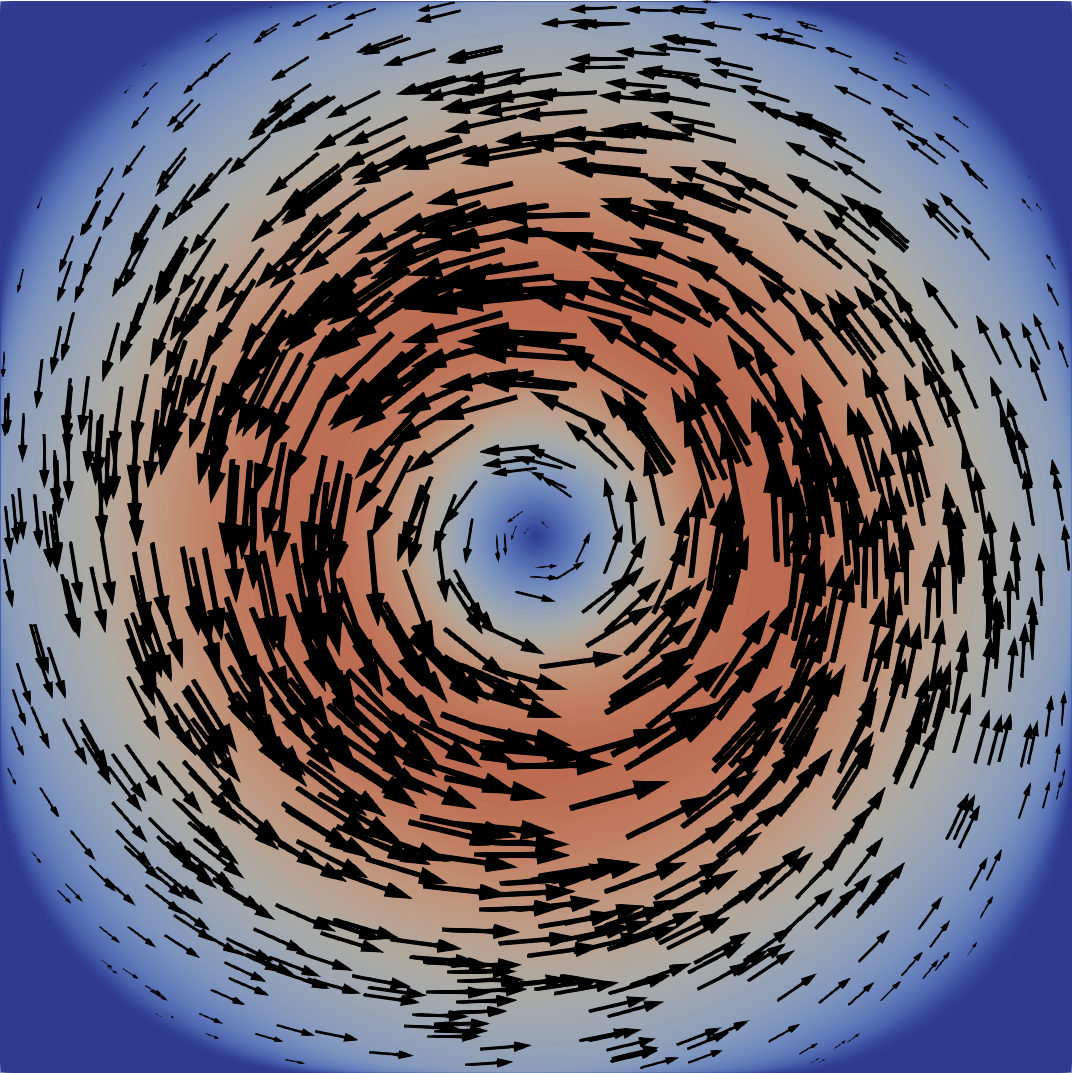} 
    &\includegraphics[align=c,height=3cm]{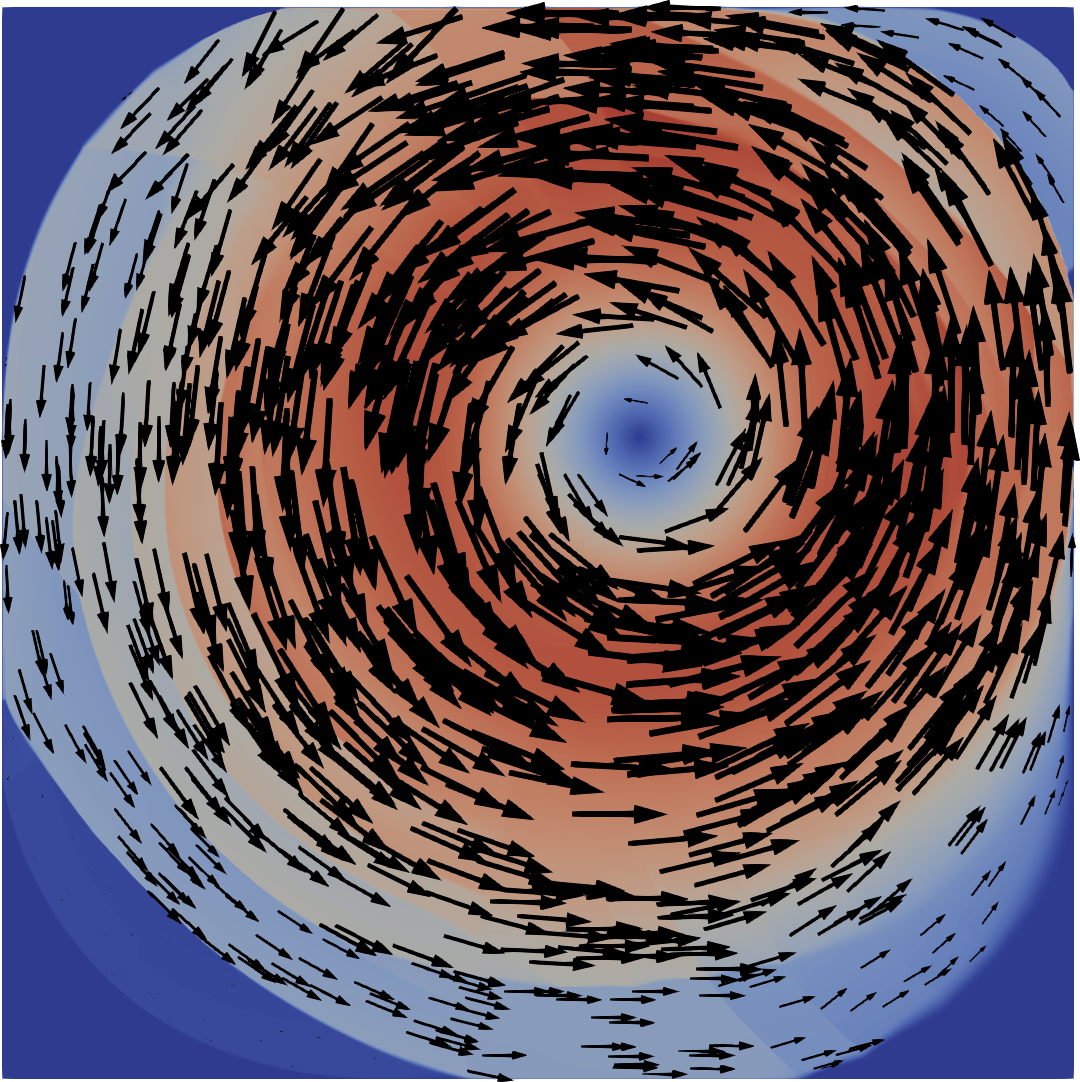}
    &\includegraphics[align=c,height=3cm]{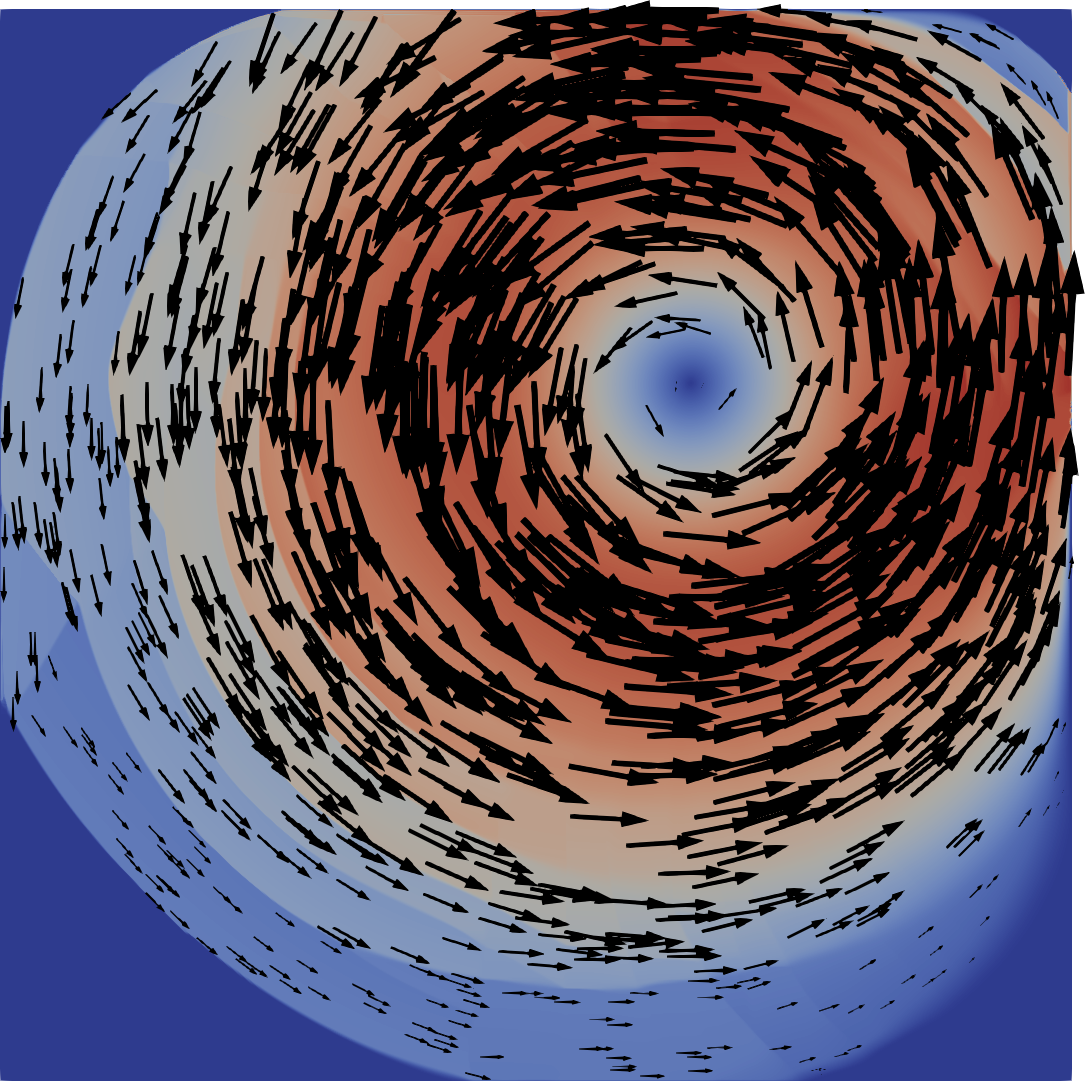}
    &\multicolumn{1}{l}{\hspace{-0.2cm}\includegraphics[align=c,height=3cm]{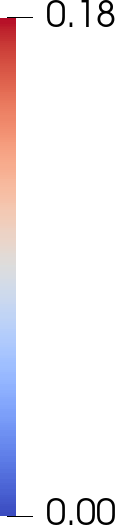}}\\[10ex]
	$T = 0$&&&$T = \SI{30}{\min}$ & $T = \SI{1}{\day}$. & $T = \SI{2}{\day}$.&
\end{tabular}
	\caption{Initial ice height $H$ and sea-ice velocity $\vel$ ($\SI{}{\m\s^{-1}}$) at different times
	$T$ for Problem II.  Arrows in the right three plots show the direction of the velocity field
  and the colors depict its magnitude. Note the discontinuities in the
  sea-ice velocity magnitude, which are due to the nonlinear rheology.
  Results are from a run on a quadrilateral mesh with mesh size
  $\Delta x = \Delta y = \SI{2}{\km}$ using \svnewton{}.}
\label{fig:vel}
\end{figure}

\begin{figure}
   \centering
	\resizebox{\textwidth}{!}{
   \begin{tabular}{p{0.2cm}cccp{0.2cm}ccl}
	&$\Delta x = \SI{4}{km}$ & \SI{2}{km} &\SI{1}{km} &&\SI{4}{km} & \SI{2}{km}\\
	\begin{tabular}{r}\rotatebox{90}{\small \gascoigne{}}\end{tabular}
	&\includegraphics[align=c,height=3.0cm]{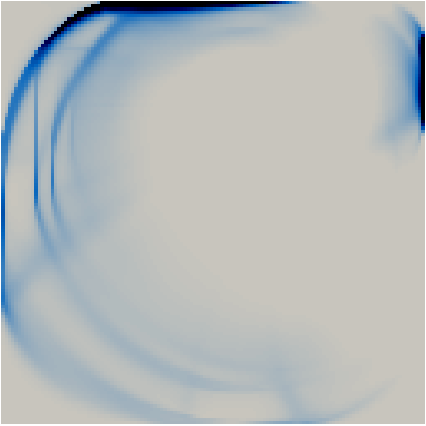}
    &\includegraphics[align=c,height=3.0cm]{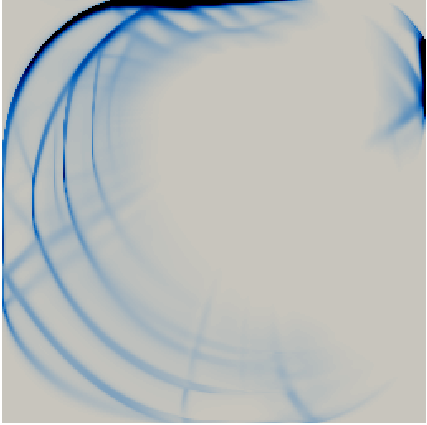}
    &\includegraphics[align=c,height=3.0cm]{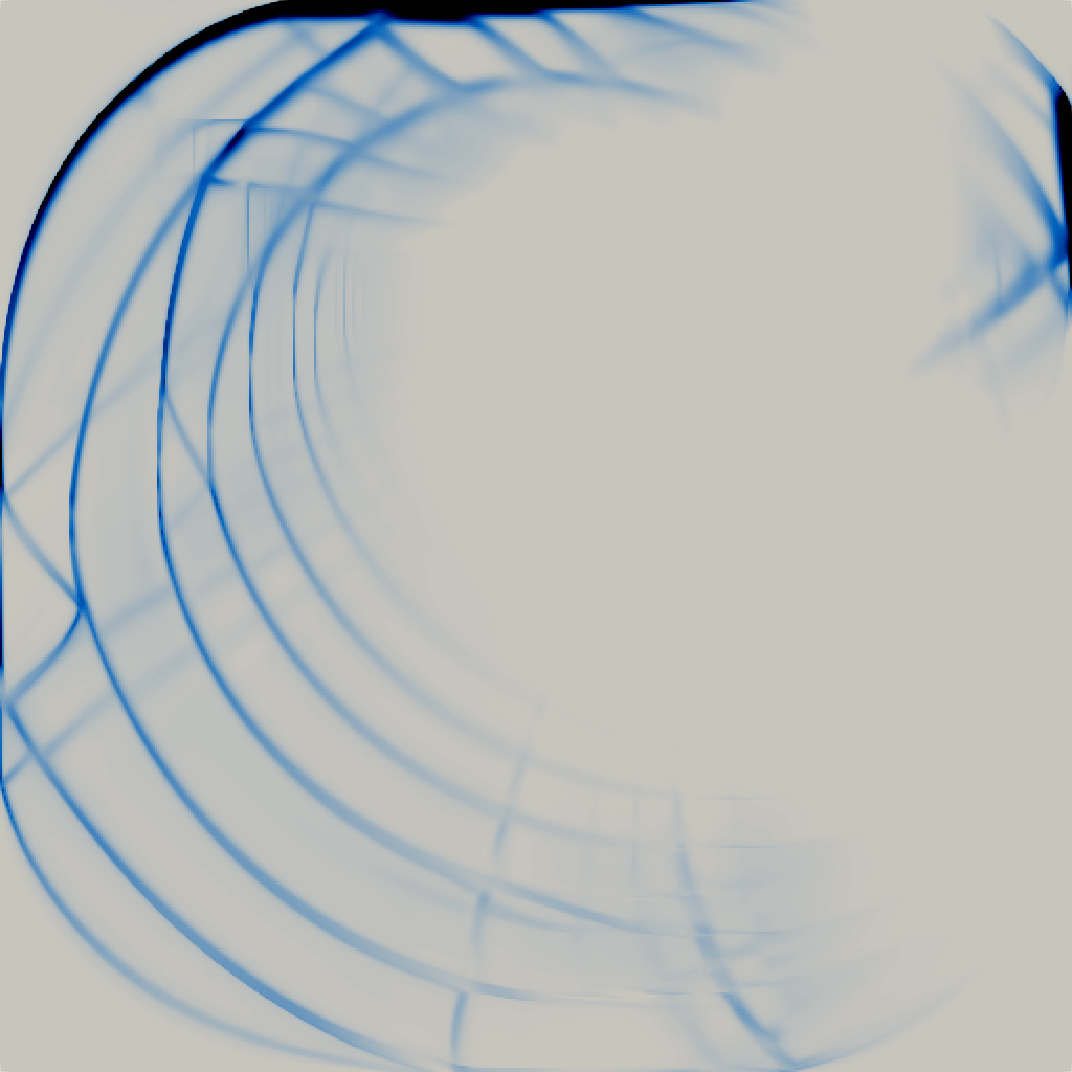}
    &\begin{tabular}{r}\rotatebox{90}{\small \mitgcm{}}\end{tabular}
    &\includegraphics[align=c,height=3.0cm]{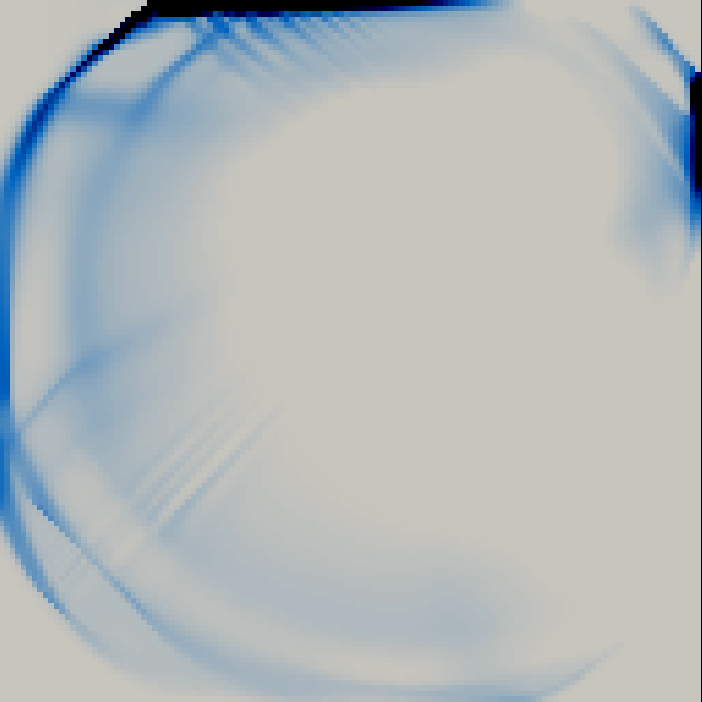}
    &\includegraphics[align=c,height=3.0cm]{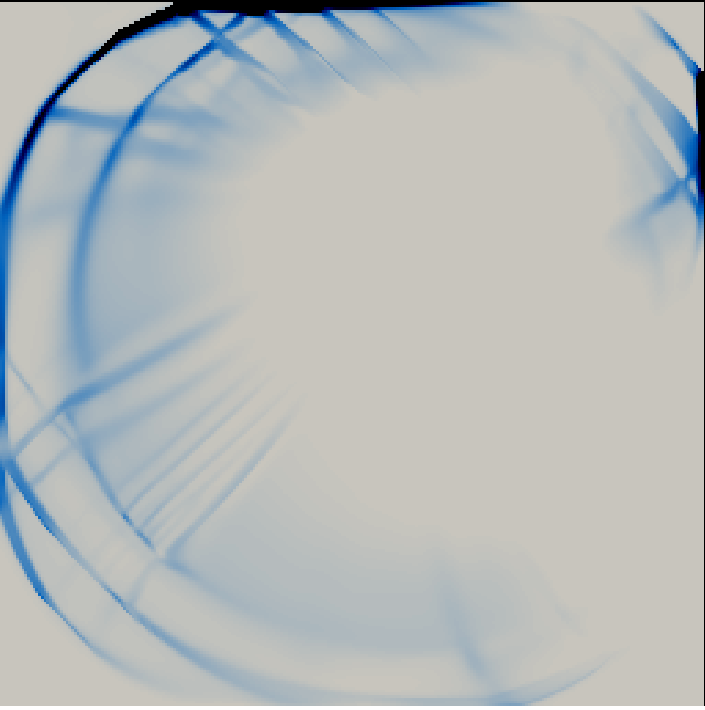}
	&\multirow{8}{*}{\hspace{-0.2cm}\includegraphics[align=c,height=3.0cm]{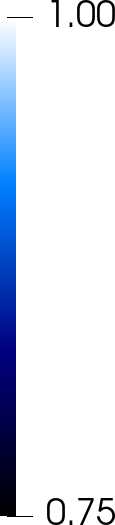}}
	\\[10ex]
	&\SI{4}{km} &\SI{2}{km} &\SI{1}{km} && \SI{2}{km} + $\mathbb{Q}_2$ for $\vel$&\SI{0.5}{km} \\
    \begin{tabular}{r}\rotatebox{90}{\small \svnewton{}}\end{tabular}
    &\includegraphics[align=c,height=3.0cm]{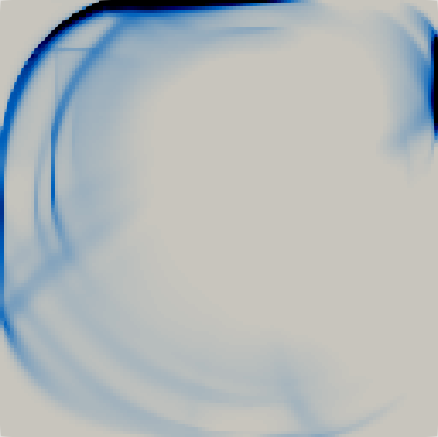}
    &\includegraphics[align=c,height=3.0cm]{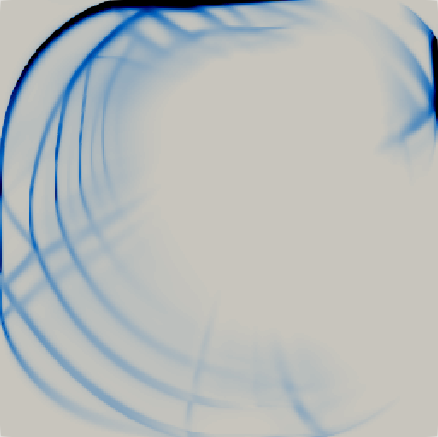}
	&\includegraphics[align=c,height=3.0cm]{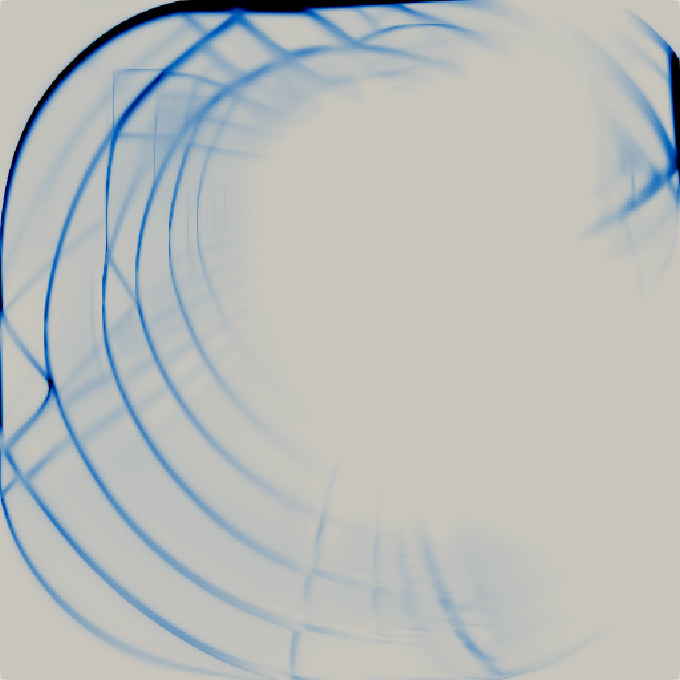}
	&&\includegraphics[align=c,height=3.0cm]{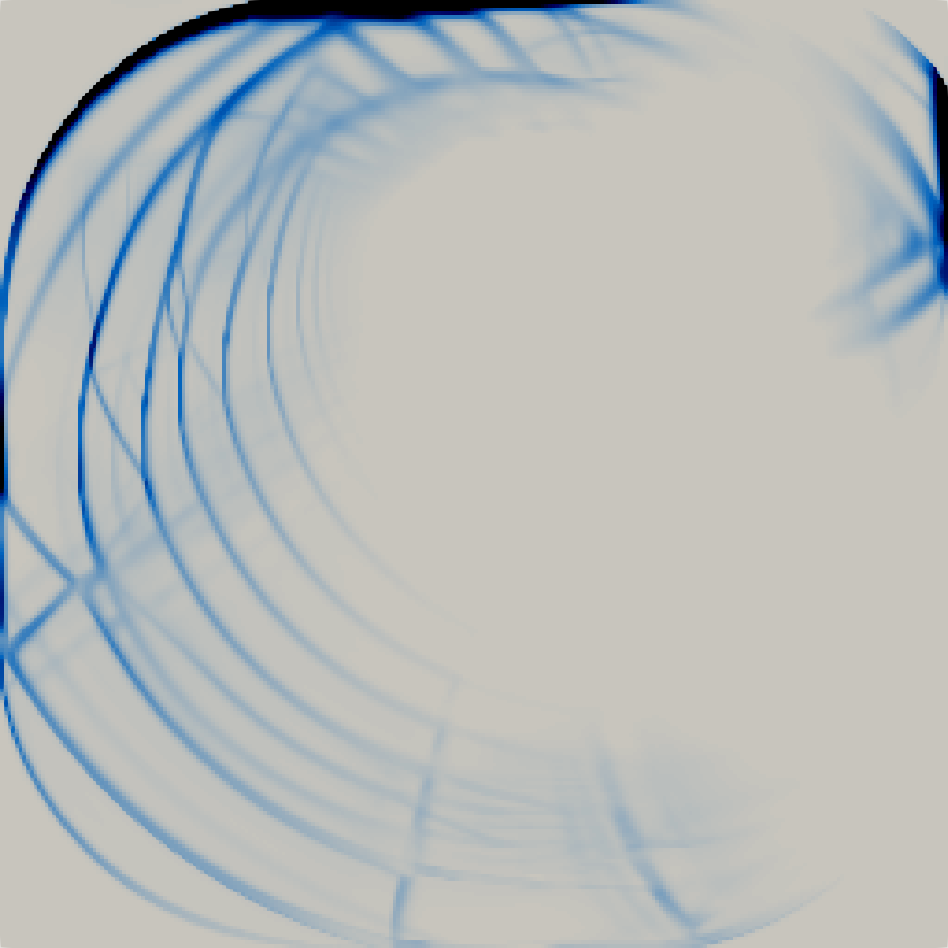}
	&\includegraphics[align=c,height=3.0cm]{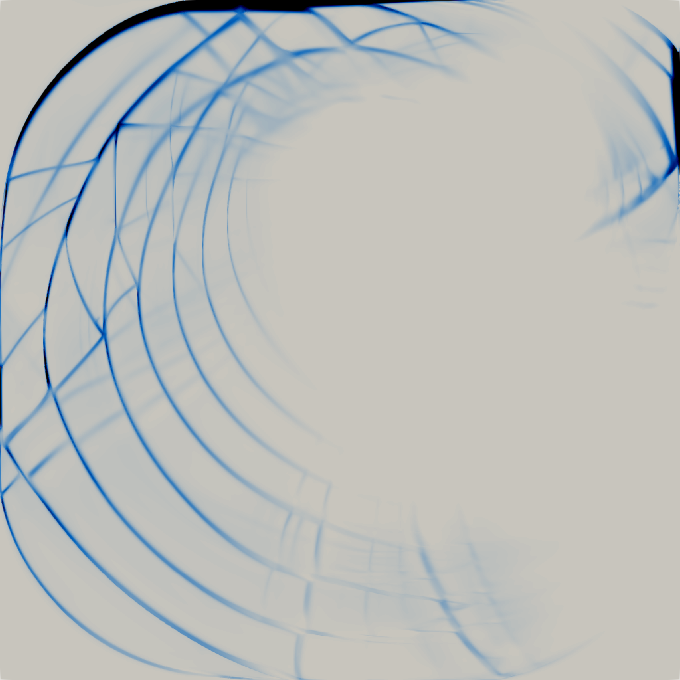}
   \end{tabular}
	}
   \caption{Sea ice concentration $A$ at 2 days for Problem II
     computed with different mesh resolutions and methods. All
     simuations use a \SI{0.5}{\hour} time step and thus $\SI{48}{\hour}/\SI{0.5}{\hour} = 96$ time steps. For the run with $\mathbb{Q}_2$ for $\vel$, we use $\mathbb{Q}_1^{\text{disc}}$ for $A$ and $H$.}
  \label{fig:iceconcen}
\end{figure}

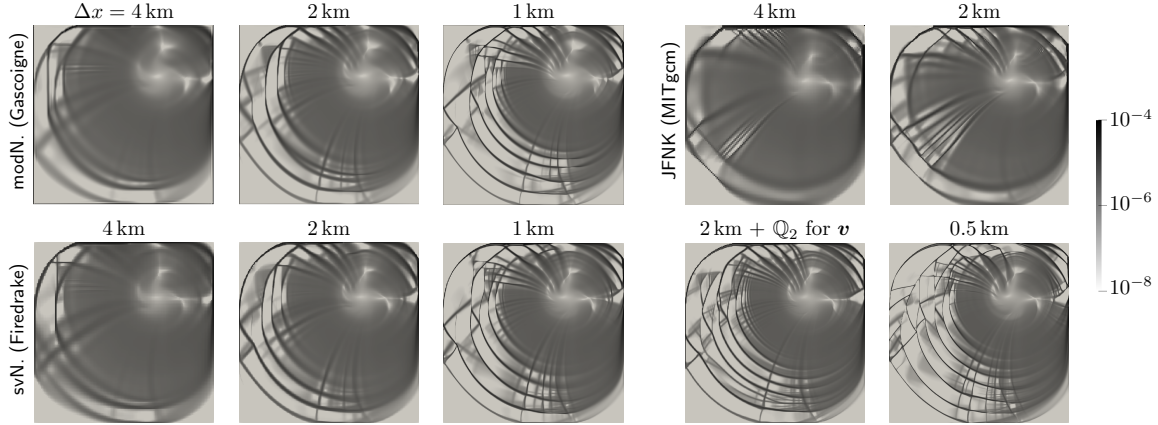
\begin{figure}
   \centering
	\resizebox{\textwidth}{!}{
   \begin{tabular}{p{0.2cm}cccp{0.2cm}ccl}
	&$\Delta x = \SI{4}{km}$ & \SI{2}{km} &\SI{1}{km} &&\SI{4}{km} & \SI{2}{km}\\
	\begin{tabular}{r}\rotatebox{90}{\small \gascoigne{}}\end{tabular}
	&\includegraphics[align=c,height=3.0cm]{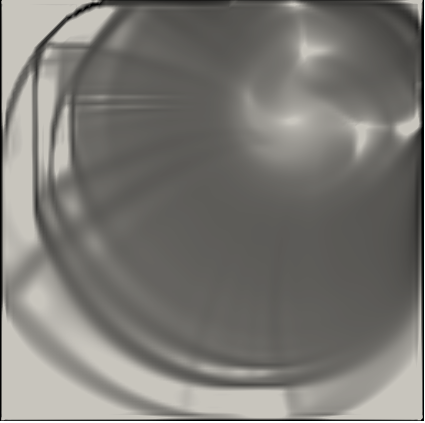}
    &\includegraphics[align=c,height=3.0cm]{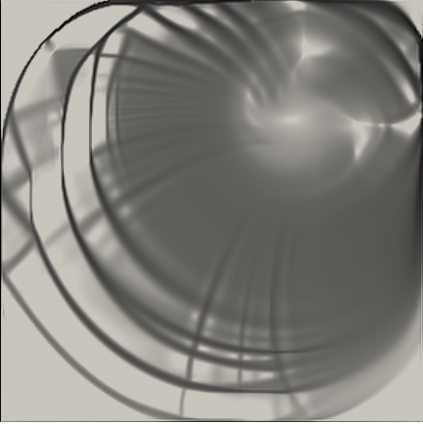}
    &\includegraphics[align=c,height=3.0cm]{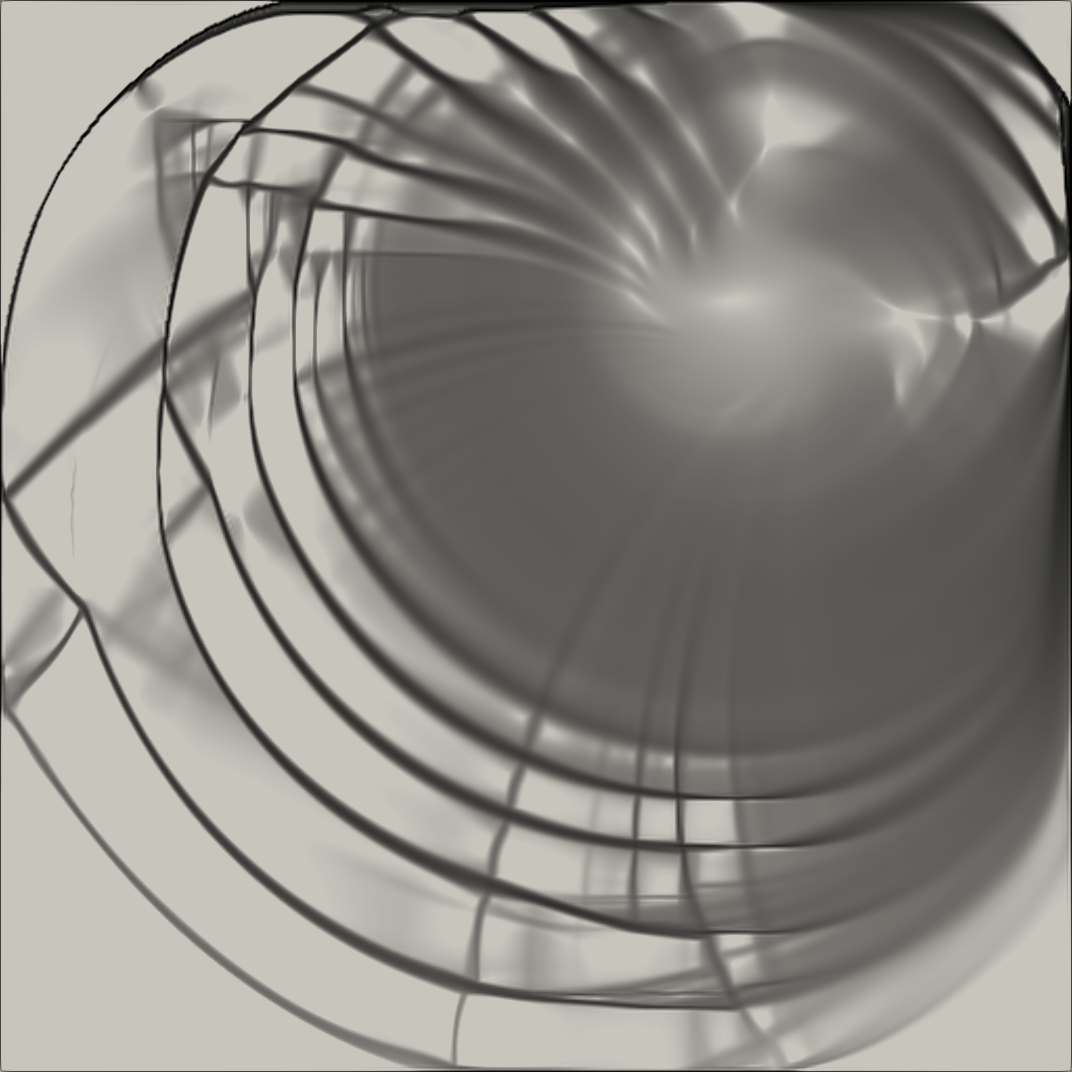}
    &\begin{tabular}{r}\rotatebox{90}{\small \mitgcm{}}\end{tabular}
    &\includegraphics[align=c,height=3.0cm]{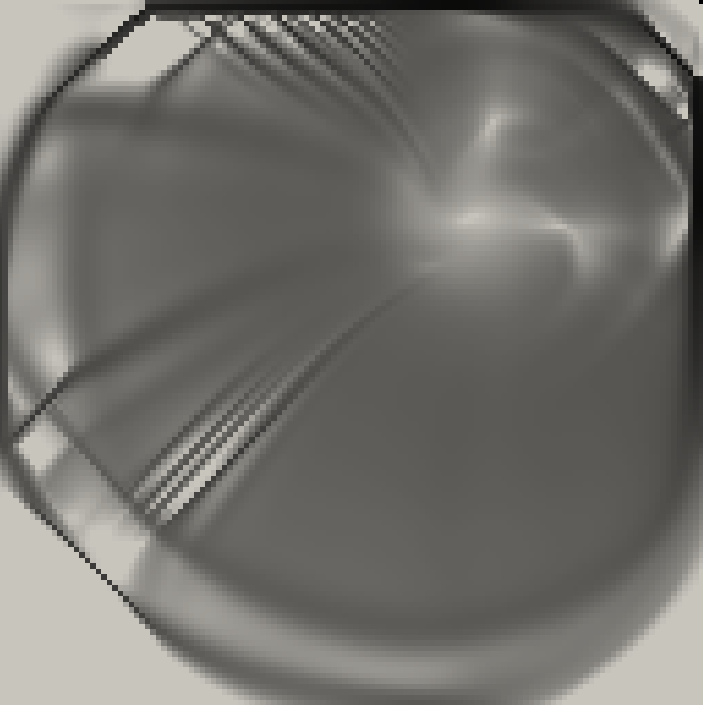}
    &\includegraphics[align=c,height=3.0cm]{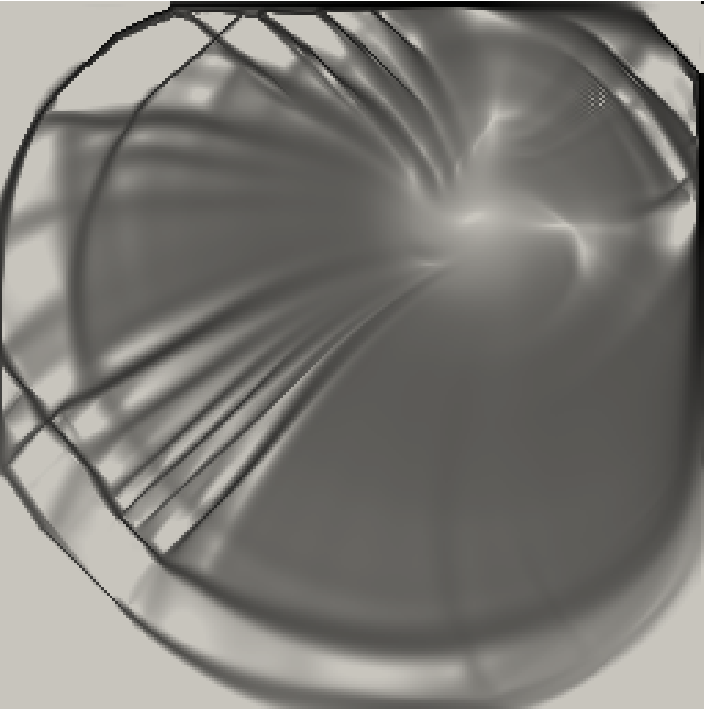}
	&\multirow{8}{*}{\hspace{-0.2cm}
	   \begin{tikzpicture}
		   \node {\includegraphics[align=c,height=3.0cm]{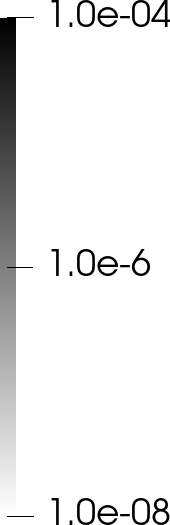}};
		   \draw [fill=white,draw=none] (-0.3,1.2)  rectangle node[pos=.5]{$10^{-4}$} (0.5,1.7)  ;
		   \draw [fill=white,draw=none] (-0.3,-0.2) rectangle node[pos=.5]{$10^{-6}$} (0.5,0.3)  ;
		   \draw [fill=white,draw=none] (-0.3,-1.6) rectangle node[pos=.5]{$10^{-8}$} (0.5,-1.1) ;
	   \end{tikzpicture}}
	\\[10ex]
	   &\SI{4}{km} &\SI{2}{km} &\SI{1}{km} && \SI{2}{km} + $\mathbb{Q}_2$ for $\vel$ &\SI{0.5}{km} \\
    \begin{tabular}{r}\rotatebox{90}{\small \svnewton{}}\end{tabular}
    &\includegraphics[align=c,height=3.0cm]{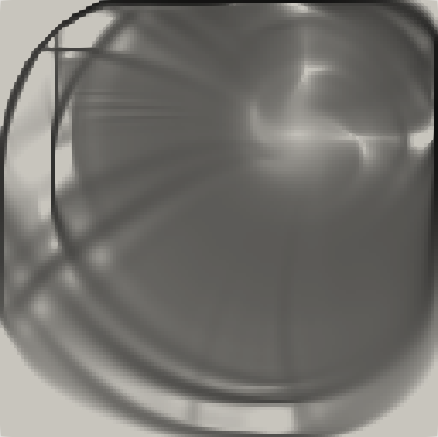}
    &\includegraphics[align=c,height=3.0cm]{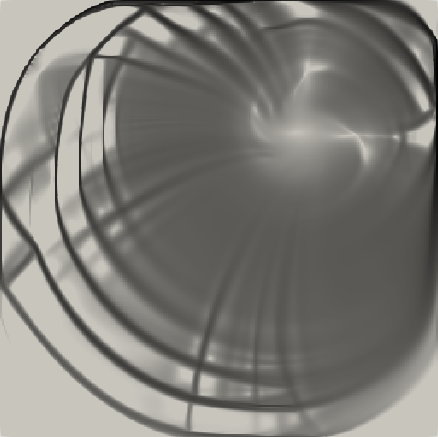}
	&\includegraphics[align=c,height=3.0cm]{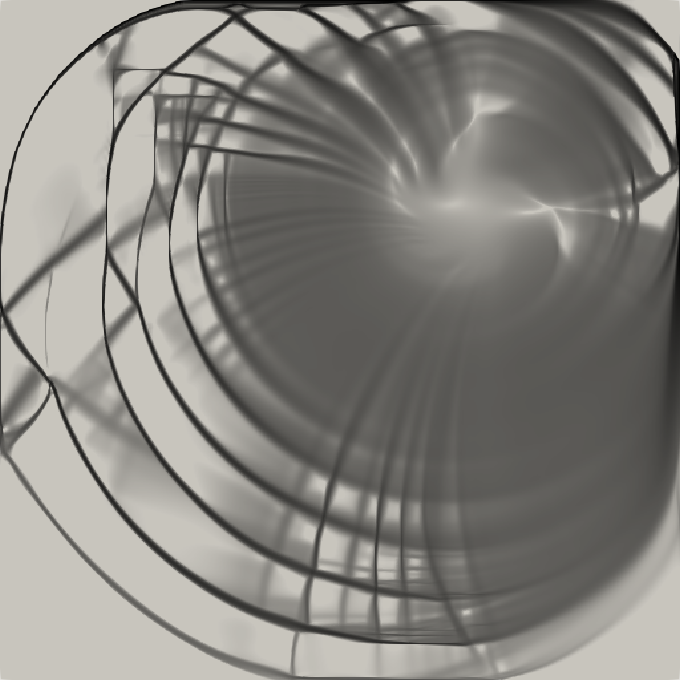}
	&&\includegraphics[align=c,height=3.0cm]{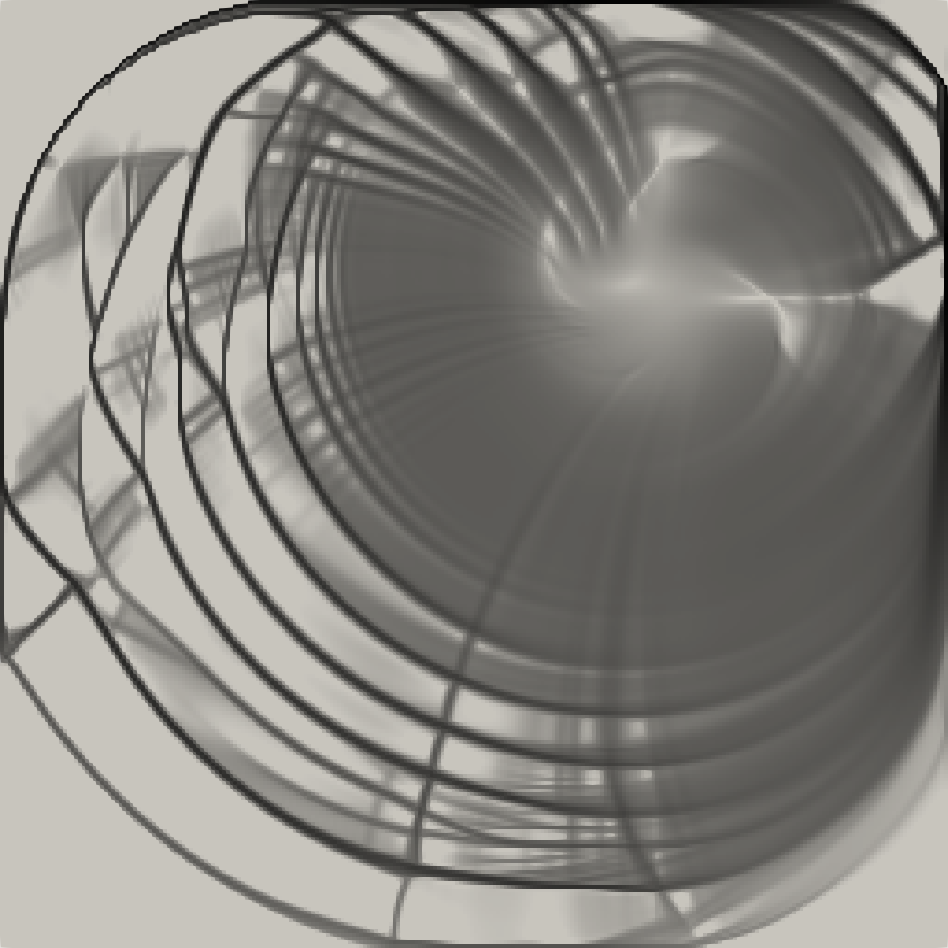}
	&\includegraphics[align=c,height=3.0cm]{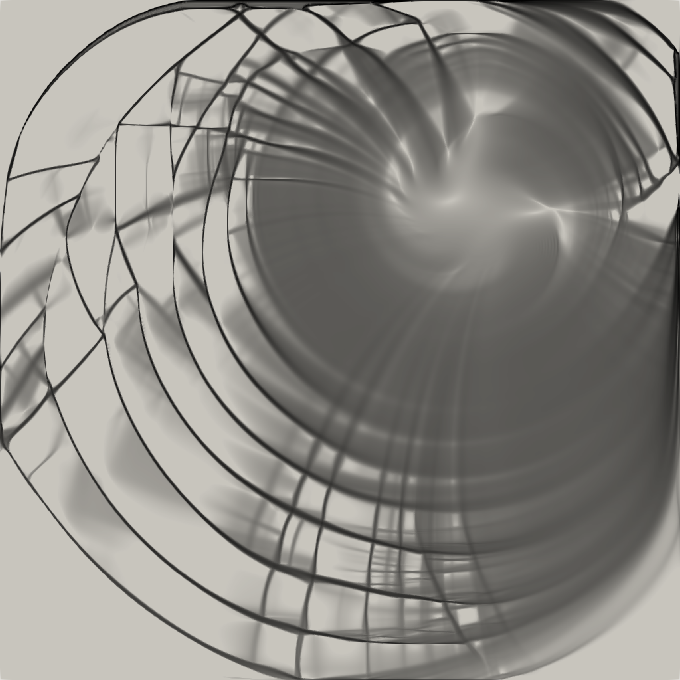}
   \end{tabular}
	}
   \caption{Shear deformation $2\sqrt{-\det \strainratetensor'}$ 
   after 2 days ($96$ time steps) for Problem II.}
  \label{fig:shear}
\end{figure}


\paragraph{Nonlinear solver convergence}
\Cref{fig:conv} and \Cref{table:meaniter} compare the performance of
different methods for the benchmark problem using a \SI{4}{\km} and a
\SI{2}{\km} mesh discretization. In particular, in \Cref{fig:conv}, we
show the number of Newton iterations required for each time step, and
in \Cref{table:meaniter}, we show the number of average Newton iterations over all
time steps. In this table, we also include a run with the smaller
$\Delta_{\min}=2\times 10^{-10}$, the constant that prevents
division by zero in \eqref{eq:sigma}, for which only the
\svnewton{} converges.
These results highlight the robustness of the \svnewton{}, 
whose average  number of iterations is more robust to mesh
refinement and the parameter $\Delta_{\text{min}}$.
Convergence failures, i.e., no reduction of the nonlinear residual by
4 orders of magnitude after 200 iterations, occur for \stdnewton{}
when we refine the mesh and/or use a smaller
$\Delta_{\text{min}}$. A similar behavior is observed for 
\mitgcm{} for the smaller mesh discretization.  The
\gascoigne{} performs comparably with the
\svnewton{}, except between day 4 and 5, which is when the
atmospheric velocity $\vel_a$ changes its direction, making the
problem particularly difficult to solve.  It however fails when
$\Delta_\text{min}=2\times 10^{-10}$.
\begin{table}
	\centering
	\caption{Average number of Newton iterations over an 4-day
          simulation of Problem II. The number of time steps, i.e.,
          the total number of the nonlinear solves for all runs is
          384. A dash indicates that the run failed, i.e., for at
          least one time step, the nonlinear solver did not converge
          after 200 iterations.}
	\label{table:meaniter}
	\begin{tabular}{r|cccc|c}
		\toprule
		$\Delta_{\min}$ & \multicolumn{4}{c|}{$2\times 10^{-9}$} & $2\times 10^{-10}$\\\midrule
		Mesh size& \multicolumn{1}{c}{\SI{4}{\km}} & 
		\multicolumn{1}{c}{\SI{2}{\km}}&\multicolumn{1}{c}{\SI{1}{\km}}&\multicolumn{1}{c|}{\SI{0.5}{\km}}&\multicolumn{1}{c}{\SI{4}{\km}}\\\midrule
		\stdnewton & $52.75$ & $78.77$ & $162.21$ & --- & --- \\
		
		\svnewton{} & $6.66$ & $11.32$ & $15.91$ & $20.20$ &$10.15$\\
		\gascoigne{} & $10.35$ & $15.87$ & $26.51$ & {\footnotesize N/A}& --- \\
		\mitgcm{} & $21.34$ &---& {\footnotesize N/A}&{\footnotesize N/A}& {\footnotesize N/A}\\
		\bottomrule
	\end{tabular}
\end{table}
\begin{figure}

\centering
\begin{tikzpicture}[scale=0.9]
\centering
\begin{axis}[
  title={},
  grid = major,
  xmin = -0.2,
  xmax = 8.2,
  ymax = 210,
  width = 15cm,
  height = 6cm,
  ylabel=\# nonlinear iter.,
  ylabel near ticks,
  xmajorticks=false]
  \addplot[dashed,black,no markers] coordinates {(-0.2,200) (8.2,200)};
  \addplot[blue,mark=*, mark size=1pt] table[col sep=comma, x index=0, y index=1] {\data/4km.csv};
  \addplot[violet,mark=*, mark size=1pt] table[col sep=space, x expr=\thisrowno{0}*2.08333E-02, y index=1] {\data/newton_iterations_4km.txt};
  \addplot[green!80!black,mark=*, mark size=1pt] table[col sep=comma, x index=0, y index=1] {\data/Newton_DDD_4km_new.csv};
  \addplot[orange,mark=*, mark size=1pt] table[col sep=comma, x index=0, y index=2] {\data/4km.csv};
  \node[rectangle, fill=gray!50!white] (A) at (0, 190)[right]{$\Delta x = 4$ km; $\Delta_{\text{min}}=2\times 10^{-9}$};
  \node[rectangle] at (59, 180)[right]{\small{\textcolor{blue}{\textbf{Failed (5.73 day)}}}};
\end{axis}
\begin{axis}[
   title={},
   grid = major,
   xmin = -0.2,
   xmax = 8.2,
   ymax = 210,
   xlabel=day,
   width = 15cm,
   height = 6cm,
   ylabel=\# nonlinear iter.,
   ylabel near ticks,
   legend style={at={(0.5, 2.3)},
                 anchor=north,
   /tikz/every even column/.append style={column sep=0.8cm},
   legend columns=3,
   draw=none},
   legend cell align={left},
   yshift=-4.5cm]

   \addplot[blue,mark=*, mark size=1pt] table[col sep=comma, x index=0, y index=1] {\data/2km.csv};
   \addlegendentry{\stdnewton{}}
   \addplot[orange,mark=*, mark size=1pt] table[col sep=comma, x index=0, y index=2] {\data/2km.csv};
   \addlegendentry{\svnewton{}}
   \addplot[dashed,black,no markers] coordinates {(-0.2,200) (8.2,200)};
   \addlegendentry{200 iterations.}
   \addplot[violet!80!white,mark=*, mark size=1pt] table[col sep=space, x expr=\thisrowno{0}*2.08333E-02, y index=1] {\data/newton_iterations_2km.txt};
   \addlegendentry{\mitgcm{}}
   \addplot[green!80!black,mark=*, mark size=1pt] table[col sep=comma, x index=0, y index=1] {\data/Newton_DDD_2km_new.csv};
   \addlegendentry{\gascoigne{}}
   \node[rectangle] (A) at (48, 180)[right]{\small{\textcolor{blue}{\textbf{Failed (4.5 day)}}}};
   \addplot[only marks,blue,mark=*,mark size=1pt,mark indices={216}] table[col sep=comma, x index=0, y index=1]{\data/2km.csv};
   
   \node[rectangle, fill=gray!50!white] (A) at (0, 190)[right]{$\Delta x = 2$ km; $\Delta_{\text{min}}=2\times 10^{-9}$};
\end{axis}
\end{tikzpicture}
\caption{Convergence behavior of different Newton-type solvers for
  Problem II with mesh sizes $\Delta x=\SI{4}{\km}$ and $\Delta
  x=\SI{2}{\km}$. Shown is the number of Newton iterations to reduce
  the nonlinear residual by a factor of $10^4$ at each time step. For
  data from \mitgcm{}, the simulation continues even when the
  residual is not sufficiently decreased after $200$ iterations.
}
\label{fig:conv}
\end{figure}
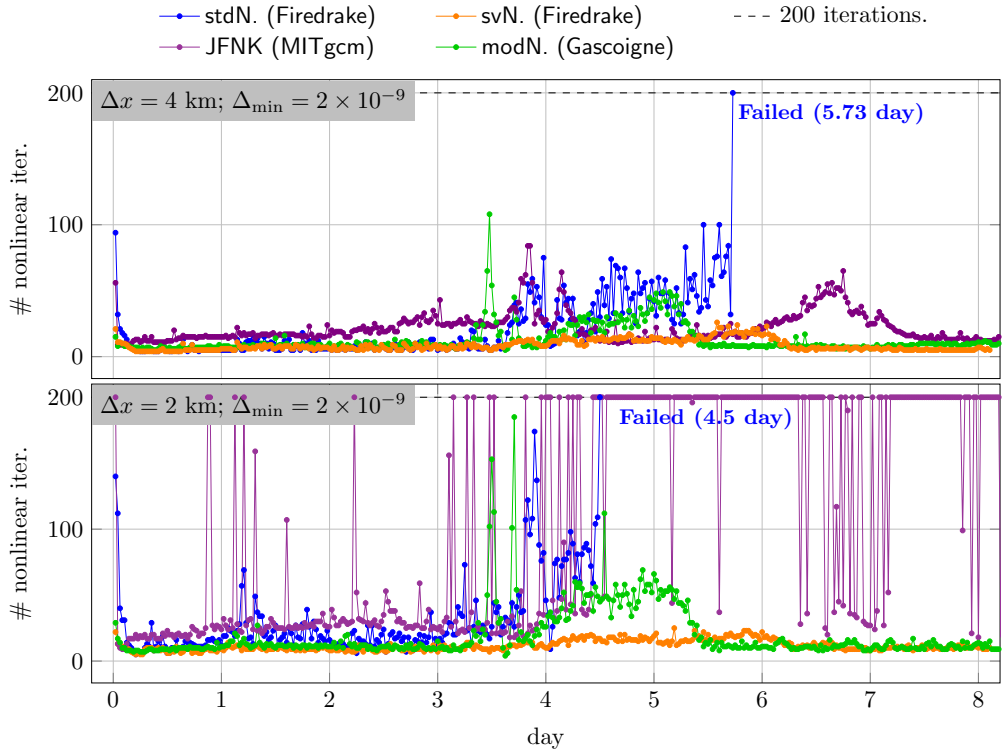

\subsection{Iterative solution of linearized systems}\label{sec:linear} 
Sea-ice simulations can easily have
millions of spatial unknowns, and thus it might be infeasible to use
direct solvers for the linearized systems arising in each Newton
step. Several iterative approaches, typically preconditioned Krylov
methods, have been proposed to solve the arising linear systems. For
instance, a line successive over-relaxation method (LSOR) as
preconditioner is proposed in the context of a Jacobian-free
Newton-Krylov solver in \cite{Lemieux2010}.  In \cite{Losch2014}, the
authors suggest to replace the computationally expensive LSOR
preconditioner by an incomplete LU-factorization (ILU) preconditioner,
in particular for high spatial resolutions. The performance of a
geometrical multigrid (GMG) method as a preconditioner is compared to
an ILU preconditioner in \cite{MehlmannRichter2016mg}. Multigrid
preconditioning is shown to substantially reduce the computational
cost and decreased iteration counts by 80\% compared to the ILU preconditioner. 
However, GMG requires a hierarchy of
meshes, which is in general not available in climate simulations. This
motivates the use of algebraic multigrid (AMG) preconditioners, which
we discuss next.  

Note that each Newton linearization of the momentum equation is a
symmetric elliptic vector system with strongly varying and anisotropic
coefficients, which are a result of the severe nonlinearity of the
 constitutive relation. For the standard Newton linearization, symmetry and positive-definiteness has been shown in \cite{MehlmannRichter2016newton} by reformulating the momentum equation to \eqref{eq:standardNewton}
and follows directly from the underlying energy minimization formulation
\eqref{eq:energy}. For the stress--velocity Newton method, symmetry and
ellipticity are not automatically satisfied due to the introduction of the
variable $\bs \pi$. Hence, positive-definiteness
and symmetry of the linearization are enforced through the modification
of the Jacobian given in \eqref{eq:modification}. This makes the
system suitable for iterative solvers that have been successfully used
for elliptic systems, such as AMG.

In particular, we use the flexible generalized minimum residual
(FGMRES) method, preconditioned with AMG. As AMG
implementation, we use the parallel implementation BoomerAMG
\cite{RugeStuben} through PETSc's interface to Hypre \cite{FalgoutYang02}. We mostly use
default parameters in BoomerAMG including a symmetric successive
over-relaxation smoother, Gaussian elimination on the coarsest level,
Falgout coarsening and a V-cycle. However, we made two changes that
are necessary due to the strongly varying coefficients, namely we
choose $3$ smoothing steps on each level 
and we set the strong threshold to $0.5$
(compared to the usual value of $0.25$ for two-dimensional
problems). The latter choice results in a smaller number of nonzero
entries for the coarser AMG matrices. The reduced fill-in helps
control the operator complexity and results in faster V-cycles without
degrading the preconditioning performance.


On the left of \Cref{fig:amg}, we show the number of Krylov iterations
required to solve the (stress-velocity) Newton linearizations arising
in the momentum equation at the first time step of Problem II. We either use AMG or
an ILU preconditioner and show results for various mesh
resolutions $\Delta x$. We find that on the coarsest mesh, i.e., $\Delta x = \SI{4}{\km}$, 
the two preconditioners behave similarly. 
On finer meshes, 
AMG outperforms the ILU preconditioner significantly. In fact, FGMRES with ILU preconditioning fails to converge within 300 iterations for meshes with resolution higher than \SI{4}{\km} (not all shown in the figure). Note that the preconditioning performance of ILU could  possibly be improved by reordering the unknowns. 
The number of Krylov iterations needed when using the AMG
preconditioner increases moderately when we refine the mesh. The same behavior has been observed by using a GMG preconditioner in \cite{MehlmannRichter2016mg}.
Such an increase is typically not observed for multigrid preconditioners
for simpler elliptic problems. The reason for the behavior
here is that upon mesh refinement, the VP constitutive
relation leads to smaller-scale features and thus steeper
gradients in the solution fields as can be seen from \Cref{fig:shear}. Thus, finer
discretizations resolve the nonlinearity better and thus the linear
systems become harder to solve.

To show that the behavior we observe for the first time step is
representative for later time steps, the right plot in \Cref{fig:amg}
shows the average number of AMG-preconditioned Krylov iterations for
each linear solve for all time steps for a simulation up to day
4. Note that the AMG-preconditioned Krylov solver converges robustly
for all time steps and all linear systems arising from the
stress-velocity Newton linearization. We again observe that this
average number is moderately larger on finer meshes, most likely
due to the more complex behavior that is resolved on finer meshes,
which makes the resulting systems more ill-conditioned.

\begin{figure}
\centering
\begin{tikzpicture}[scale=0.9] 
\begin{groupplot}[
group style={group size=2 by 1, horizontal sep=2cm,},
]
\nextgroupplot[
grid=major,
xlabel=nonlinear iter for 1st time step,
ylabel=$\#$ Krylov iter.,
ylabel near ticks,
title style={align=center},
legend to name={AMGLegend},legend style={legend columns=4,draw=none,
/tikz/every even column/.append style={column sep=0.4cm},
font=\small},
legend cell align={left},
ymin=0, ymax=62,
xmin=0,xmax=25]
\addplot[blue,mark=*, mark size=1pt] table[col sep=comma,x index=0, y index=1] {\data/firstmomen_iter.txt};
\addlegendentry{AMG, $4$ km}
\addplot[green!40!black,mark=square*, mark size=1pt] table[col sep=comma,x index=0, y index=2] {\data/firstmomen_iter.txt};
\addlegendentry{AMG, $2$ km}
\addplot[orange,mark=triangle*, mark size=1.8pt] table[col sep=comma,x index=0, y index=3] {\data/firstmomen_iter.txt};
\addlegendentry{AMG, $1$ km}
\addplot[violet,mark=diamond*, mark size=1.8pt] table[col sep=comma,x index=0, y index=4] {\data/firstmomen_iter.txt};
\addlegendentry{AMG, $0.5$ km}
\addplot[blue!70!white,mark=*, dashed, mark size=1pt, line width=1pt,
  mark options=solid] table[col sep=comma,x index=0, y index=5] {\data/firstmomen_iter.txt};
\addlegendentry{ILU(2), $4$ km}
\addplot[green!80!black,dashed,mark=square*, mark size=1pt, line
  width=1pt, mark options=solid] table[col sep=comma,x index=0, y index=6] {\data/firstmomen_iter.txt};
\addlegendentry{ILU(2), $2$ km}
\addplot[orange!70!white,dashed,mark=triangle*, mark size=1.8pt, line
  width=1pt, mark options=solid] table[col sep=comma,x index=0, y index=7] {\data/firstmomen_iter.txt};
\addplot[dashed,black,no markers] coordinates {(-0,300) (26,300)};

\nextgroupplot[
grid=major,
xlabel=day,
ylabel=avg.\ $\#$ Krylov iter.,
ymin=0,ymax=62,
ylabel near ticks,
xmin=0,xmax=4]
\addplot[blue,mark=*, mark size=1pt] table[col sep=comma, x index=4, y index=3] {\data/amg/4km_process.csv};
\addplot[green!40!black,mark=square*, mark size=1pt] table[col sep=comma, x index=4, y index=3] {\data/amg/2km_process.csv};
\addplot[orange,mark=triangle*, mark size=1.8pt] table[col sep=comma, x index=4, y index=3] {\data/amg/1km_process.csv};
\addplot[violet,mark=diamond*, mark size=1.8pt] table[col sep=comma, x index=4, y index=3] {\data/amg/05km_process.csv};

\end{groupplot}
\path (group c1r1.north east) -- node[above, yshift=0.1cm]{\ref{AMGLegend}} (group c2r1.north west);
\end{tikzpicture}
\caption{Convergence of preconditioned Krylov method for linearized
  systems for Problem II. Shown on the left is the number of
  preconditioned Krylov iterations ($y$-axis) required for each Newton
  linearization ($x$-axis) for the first time step.  Shown are results
  for differently mesh sizes $\Delta x$ using an AMG or an incomplete
  LU preconditioner. Shown on the right is the number of average
  Krylov iterations preconditioned with AMG needed for the linear
  Newton solves at each time step. Note that the number of Krylov
  iterations varies but remains moderate for all Newton linearizations
  throughout the 4-day simulation. On average, 14, 17, 22 and 30 Krylov iterations
  are needed for the meshes with $\Delta x=4 \text{km},2 \text{km},1
  \text{km},0.5 \text{km}$, respectively.}
\label{fig:amg}
\end{figure}
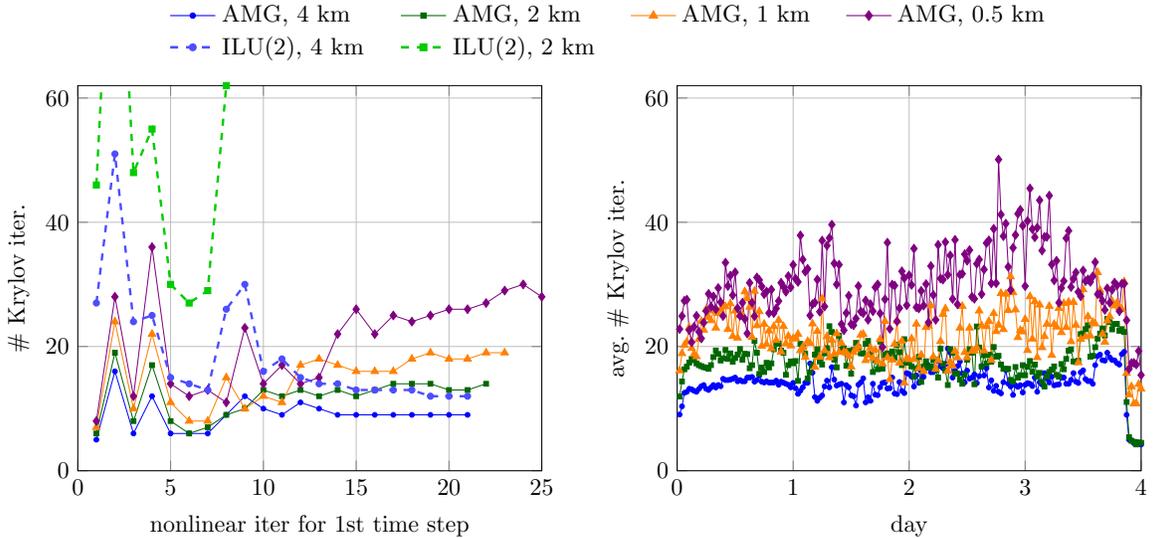

Finally, we study the weak parallel scalability of our solver and
present a comparison of timings in \Cref{fig:runtime}. In a weak
scaling study, one increases the problem size (by refining the mesh)
and, at the same time, increases the computational resources used for
solving the problem. When this is done such that the computational
work per compute core remains roughly the same, perfect weak
scalability would amount to constant run time. As baseline, we show
the time for the first momentum solve based on the standard Newton
linearization with line search and the direct parallel solver MUMPS
\cite{MUMPS01} in \Cref{fig:runtime} (blue bars). Since MUMPS's (and
any direct solver's) parallel scalability is limited, we do not go
beyond one compute node, i.e., the 56 CPU cores that share the same
physical memory. The substantial growth in runtime we observe as we go
from a \SI{4}{\km} to a \SI{0.25}{\km} mesh is primarily due to the
increasing number of Newton steps upon mesh refinement. Using the
proposed alternative Newton linearization results in a substantial
improvement due to a much more moderate increase in the number of
Newton iterations and resulting wall time (green bars).  For instance,
from \Cref{table:meaniter} we observe that the number of Newton
iteration increases from 16 for $\Delta x=\SI{1}{\km}$ to 20 for
$\Delta x=\SI{0.5}{\km}$. The growth in run time is more significant,
however, due to the limited weak parallel scalability of direct
solvers. Finally, we combine the proposed Newton linearization with
the AMG-preconditioned Krylov method to solve the linear problems
(orange bars).
First, we observe a significant redution in overall run time. The
remaining growth going from coarse to fine meshes is mainly due to the increasing number of
Krylov iterations when the mesh is refined (compare the left plot in
\Cref{fig:amg}). This increase is larger for the first momentum solve
than on average as can bee seen on the right plot in \Cref{fig:amg},
where we find that the average number of Krylov iterations goes up by
a factor of $\approx 2$ comparing $\Delta x=\SI{4}{\km}$ with $\Delta
x=\SI{0.5}{\km}$. Hence, on average, we expect a more favorable
scalability of our stress--velocity Newton linearization with
AMG-preconditioned Krylov solver than shown in \Cref{fig:runtime}. An
additional factor contributing to the increase in solve time is the
required Message Passing Interface communication through the network
due to parallelism.

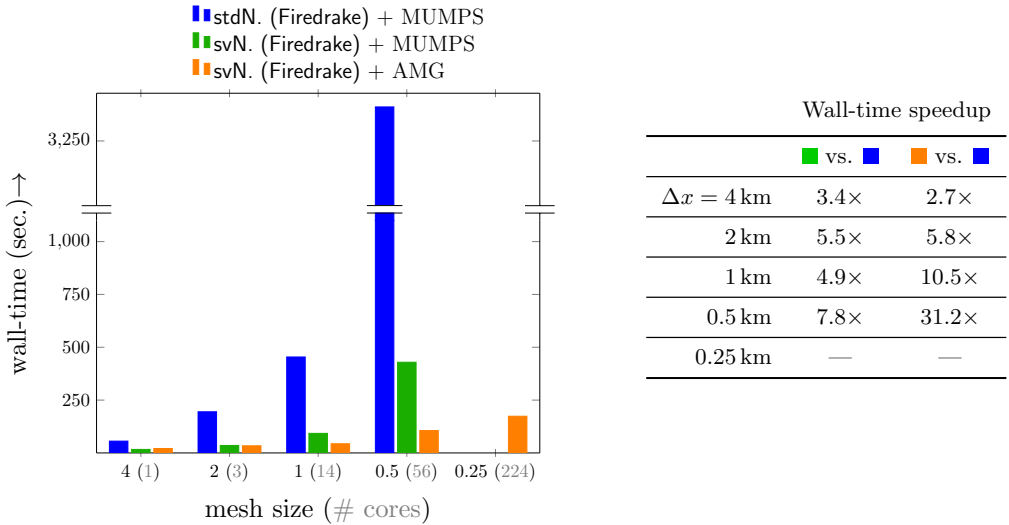
\begin{figure}
\centering
\makeatletter
\pgfplotsset{
    groupplot xlabel/.initial={},
    every groupplot x label/.style={
        at={($({\pgfplots@group@name\space c1r\pgfplots@group@rows.west}|-{\pgfplots@group@name\space c1r\pgfplots@group@rows.outer south})!0.5!({\pgfplots@group@name\space c\pgfplots@group@columns r\pgfplots@group@rows.east}|-{\pgfplots@group@name\space c\pgfplots@group@columns r\pgfplots@group@rows.outer south})$)},
        anchor=north,
    },
    groupplot ylabel/.initial={},
    every groupplot y label/.style={
            rotate=90,
        at={($({\pgfplots@group@name\space c1r1.north}-|{\pgfplots@group@name\space c1r1.outer
west})!0.5!({\pgfplots@group@name\space c1r\pgfplots@group@rows.south}-|{\pgfplots@group@name\space c1r\pgfplots@group@rows.outer west})$)},
        anchor=south
    },
    execute at end groupplot/.code={%
      \node [/pgfplots/every groupplot x label]
{\pgfkeysvalueof{/pgfplots/groupplot xlabel}};  
      \node [/pgfplots/every groupplot y label] 
{\pgfkeysvalueof{/pgfplots/groupplot ylabel}};  
    }
}

\def\endpgfplots@environment@groupplot{%
    \endpgfplots@environment@opt%
    \pgfkeys{/pgfplots/execute at end groupplot}%
    \endgroup%
}
\makeatother

\pgfplotsset{ every non boxed x axis/.append style={x axis line style=-},
     every non boxed y axis/.append style={y axis line style=-}}
\begin{tikzpicture}[baseline=(current axis.outer east),scale=0.7]
\begin{groupplot}[
    group style={
        group size=1 by 2,
		xticklabels at=edge bottom,
        vertical sep=0pt
    },
	xmin=0.5, xmax=5.5,
    width=10cm,
    groupplot ylabel={wall-time (sec.)$\rightarrow$},
    groupplot xlabel={mesh size (\textcolor{gray}{\# cores})},
]
\nextgroupplot[ymin=3000,
               ymax=3400,
               ybar,
               axis x line=top, 
               axis y discontinuity=parallel,
			   ytick={3250},
	           height=4.0cm,
              ]
\addplot[blue,fill=blue] table[col sep=comma, y index=5]{\data/amg_weak.txt};
\addplot[green,fill=green] table[col sep=comma,y index=6]{\data/amg_weak.txt};
\addplot[green,fill=green] table[col sep=comma,y index=6]{\data/amg_weak.txt};
\nextgroupplot[ymin=0,
               ymax=1100,
               ybar,
               axis x line=bottom,
               xtick = {1,2,3,4,5},
               xticklabels={$4$ (\textcolor{gray}{1}),
                $2$ (\textcolor{gray}{3}),
                $1$ (\textcolor{gray}{14}),
                $0.5$ (\textcolor{gray}{56}),
                $0.25$ (\textcolor{gray}{224})},
                height=6.0cm,
                ytick={250,500,750,1000},
                legend style={draw=none,at={(0.55, 1.95),row sep=2pt},
                              anchor=north},
                legend style={font=\large},
                legend cell align={left},
                    after end axis/.code={ 
                \draw [line width=0.14cm, white, decorate] (rel axis cs:0.1,1.05) -- (rel axis cs:0.9,1.05);
                \draw [black, thin] (rel axis cs:0.61,1.064) -- (rel axis cs:0.69,1.064);
                \draw [black, thin] (rel axis cs:0.61,1.032) -- (rel axis cs:0.69,1.032);
                }]
\addplot[blue,fill=blue] table[col sep=comma,y index=2]{\data/amg_weak.txt};
\addlegendentry{\stdnewton{} + MUMPS}
\addplot[green!65!black!90!yellow,fill=green!65!black!90!yellow] table[col sep=comma,y index=3]{\data/amg_weak.txt};
\addlegendentry{\svnewton{} + MUMPS}
\addplot[orange,fill=orange] table[col sep=comma,y index=4]{\data/amg_weak.txt};
\addlegendentry{\svnewton{} + AMG}
\end{groupplot}
\end{tikzpicture}
\hspace*{5ex}
\scalebox{0.95}{
{
\small
\begin{tabular}{rcc}
&\multicolumn{2}{c}{Wall-time speedup}\\[1ex]
\toprule
    & \begin{tikzpicture}\node[rectangle,fill=green!80!black] (r) at (0,0) {};\end{tikzpicture} vs. \begin{tikzpicture}\node[rectangle,fill=blue] (r) at (0,0) {};\end{tikzpicture} 
		& \begin{tikzpicture}\node[rectangle,fill=orange] (r) at (0,0) {};\end{tikzpicture} vs. \begin{tikzpicture}\node[rectangle,fill=blue] (r) at (0,0) {};\end{tikzpicture}\\\midrule
    $\Delta x = \SI{4}{\km}$ &  3.4$\times$ & 2.7$\times$\\\midrule
    \SI{2}{\km}&  5.5$\times$ & 5.8$\times$\\\midrule
    \SI{1}{\km}&  4.9$\times$ & 10.5$\times$\\\midrule
    \SI{0.5}{\km}&  7.8$\times$ & 31.2$\times$\\\midrule
    \SI{0.25}{\km}&  --- & --- \\\bottomrule
\end{tabular}
}}
\caption{Weak scalability on Texas Advanced Computing Center's
  Frontera (Intel Cascade Lake nodes) for first momentum solve in
  Problem II. Shown on the left are the wall clock times for a weak
  parallel scaling study.  Shown in brackets is the number of Message
  Passing Interface (MPI) processes chosen so that the number of
  unknowns per core are around $32$K. Timings for three
  implementations are compared, namely a standard Newton method (blue)
  and the proposed velocity-stress Newton method (green), both using
  the direct linear solver MUMPS \cite{MUMPS01}. Shown in orange are
  timings for the velocity-stress Newton method with an
  AMG-preconditioned Krylov method as iterative linear solver. For
  mesh resolution of \SI{0.25}{\km}, only the solver proposed in this
  paper can be used as the problem has 8.4 million unknowns. The table
  on the right compares the speedup factors for the different
  methods.}
\label{fig:runtime}
\end{figure}

\section{Conclusion}
We have shown that a novel linearization of the momentum equation
arising in the commonly used VP sea-ice model results in
faster and more robust Newton convergence, with iteration numbers that
only increase moderately upon mesh refinement (see the table in \Cref{fig:cond} and \Cref{table:meaniter}).
The approach is motivated by the
observation that the implicit time step of the momentum equation can
also be written as an energy minimization problem.
This suggested the alternative Newton linearization presented here,
which is similar to techniques used in primal-dual optimization
algorithms. We hope that these ideas will eventually be adopted in the
sea-ice community as they substantially reduce the computational cost
for implicitly time-stepped VP models and thus enable finer mesh
resolutions and better resolved models. More generally, the approach
to lift nonlinear equations to a higher dimensional space before
Newton linearization could be useful for other challenging systems as
well.


Note that our derivation of the alternative Newton step assumes an
elliptic yield curve, i.e., the most common yield criterion for
viscous-plastic sea-ice models. Other yield curves are possible and
scientifically relevant
(e.g., \cite{Ringeisen21, Zhang21}) and it remains to be seen if our
approach generalizes to these different constitutive relationships.

We additionally study iterative parallel solvers for the
large and ill-conditioned (but positive definite) systems arising upon
Newton linearizations. We propose and study AMG
as preconditioner for a Krylov method to solve these systems.  
Computational efficiency is a strong constraint for sea-ice models in
the climate modeling context. Any new method will need to demonstrate
that it can meet this requirement so that accurate sea ice dynamics
solutions become available at reasonable cost. The computational cost
of the primal-dual Newton-Krylov solver with AMG preconditioning
increases only moderately with grid refinement. Further, AMG does not
require a mesh hierarchy and efficient parallel open-source AMG
libraries are available, so that this combination of methods may be a
good candidate for solving the momentum equation of sea-ice models in
particular, and for climate simulations in general.

\section*{Acknowledgements}
\noindent
This work has been supported by the Multidisciplinary University
Research Initiatives (MURI) Program, Office of Naval Research (ONR)
grant number N00014-19-1-2421, and by the DFG priority program
``Antarctic Research with Comparative Investigations in Arctic Ice
Areas'' project number 463061012.

\bibliography{extra,seaice,firedrake}
\end{document}